\documentclass{amsart} 
\usepackage{amsmath,amsthm,amssymb,amsfonts,amscd}
\usepackage{xcolor}
\usepackage[dvips]{graphicx}
\usepackage[all]{xy}
\usepackage{xcolor}

\newtheorem{theorem}{Theorem}[section]
\newtheorem{lemma}[theorem]{Lemma}
\newtheorem{corollary}[theorem]{Corollary}

\newtheorem{example}[theorem]{Example}

\theoremstyle{definition}
\newtheorem{definition}[theorem]{Definition}

\newtheorem{remark}[theorem]{Remark}

\numberwithin{equation}{section}

\begin{document}

\title[Fourier--Feynman transforms and convolution type operations] 
{Analytic Fourier--Feynman transforms and convolution type operations 
associated with Gaussian processes  on Wiener space}


\author{Jae Gil Choi}
\address{Department of Mathematics, Dankook University, Cheonan 330-714, Korea}
\email{jgchoi@dankook.ac.kr}

\author{Seung Jun Chang$^*$}
\address{Department of Mathematics, Dankook University, Cheonan 330-714, Korea}
\email{sejchang@dankook.ac.kr}
   
\thanks{$^*$ Corresponding author.}


 

\begin{abstract} 
In this paper we introduce the concept of a convolution type operation of functionals 
on Wiener space. It contains several kinds of the concepts of convolution products on  
Wiener space,  which have been studied by many authors. We then investigate fundamental 
relationships between generalized analytic  Fourier--Feynman transforms and  convolution type 
operations. Both of the generalized analytic Fourier--Feynman transform of  the convolution 
type operation and the convolution type operation of the  generalized analytic Fourier--Feynman 
transforms are represented as a product of the generalized analytic Fourier--Feynman transforms.
\end{abstract}

\maketitle

\setcounter{equation}{0}
\section{Introduction}\label{sec:into}

\par
For $f\in L_1(\mathbb R^n)$, let the Fourier transform $\mathcal F(f)$ 
of $f$ be given by
\[
\mathcal{F}(f)(\vec u)
=\int_{\mathbb R^n}e^{i\vec u\cdot \vec v}f(\vec v)dm_L^n(\vec v)
\]
where $dm_L^n (\vec v)$ denotes the normalized Lebesgue measure 
$(2\pi)^{-n/2}d\vec v$ on $\mathbb R^n$. Also, for $f, g\in L_1(\mathbb R^n)$, 
let the  convolution $f*g$ of $f$ and $g$ be given by
\[
(f*g)(\vec u)
=\int_{\mathbb R^n} f(\vec u-\vec v)g(\vec v)dm_L^n(\vec v).
\]
Then the Fourier transform $\mathcal F$ acts like a group homomorphism  
with convolution $*$ and ordinary multiplication  on $L_1(\mathbb R^n)$.
More precisely, one can see that  for $f, g \in L_1(\mathbb R^n)$
\[
\mathcal{F}(f*g)=\mathcal{F}(f)\mathcal{F}(g).
\]

\par
Let $C_0[0,T]$ denote one-parameter Wiener space; that is, the space of all  
real-valued continuous functions $x$ on $[0,T]$ with $x(0)=0$. Let $\mathcal{M}$ 
denote the class of all Wiener measurable subsets of $C_0[0,T]$ and let $m$ be the
Wiener measure.
Then $(C_0[0,T],\mathcal{M},m)$ is a complete measure space.

\par 
A subset $B$ of $C_0[0,T]$ is said to be 
scale-invariant measurable (s.i.m.) (see \cite{JS79-I}) provided 
$\rho B\in \mathcal{M}$ for all $\rho>0$,
and a scale-invariant measurable set $N$ 
is said to be scale-invariant null provided  
$m(\rho N)=0$ for all $\rho>0$.
A property that holds except on  a scale-invariant  
null set is said to be hold  scale-invariant almost 
everywhere (s-a.e.).
If two functionals $F$ and $G$ are equal s-a.e., 
we write $F\approx G$.

The concept of the analytic Fourier--Feynman transform(FFT) on the Wiener 
space $C_0[0,T]$, initiated by Brue \cite{Brue}, has been developed in 
the literature. This transform  and its properties are similar in many 
respects to the ordinary Fourier function transform. For an elementary 
introduction to the analytic FFT, see \cite{SS04} 
and the references cited therein.
First of all, we refer to \cite{SS04} for    the precise 
definitions and the notations of  
the analytic FFT and the convolution product(CP)
on the  Wiener space $C_0[0,T]$. In \cite{HPS95}, Huffman, 
Park and Skoug defined a CP for functionals on $C_0[0,T]$, 
and they then obtained various results for the  analytic FFT
and the  CP \cite{HPS96,HPS97-1,HPS97-2}.  
In previous researches involving \cite{CCKSY05,HPS95,HPS96,HPS97-1,HPS97-2}, 
the authors have been established the relationship between the analytic 
FFT and the corresponding CP of 
functionals $F$ and $G$ on $C_0[0,T]$, in the form 
\begin{equation}\label{eq:intro}
T_q\big((F*G)_q\big)(y)
=
T_q(F)\bigg(\frac{y}{\sqrt2}\bigg)T_q(G)\bigg(\frac{y}{\sqrt2}\bigg) 
\end{equation}
for s-a.e. $y \in C_0[0,T]$.

An essential  structure  hidden in the proof of equation \eqref{eq:intro}
is based on the fact that the Gaussian processes 
\[
\mathfrak Z_{+} \equiv \bigg\{\frac{x_1+x_2}{\sqrt2}: x_1,x_2\in C_0[0,T]\bigg\}
\mbox{ and }
\mathfrak Z_{-} \equiv \bigg\{\frac{x_1-x_2}{\sqrt2}: x_1,x_2\in C_0[0,T]\bigg\}
\]
are independent, and the processes $\mathcal Z_{+}$ and $\mathcal Z_{-}$
are equivalent to the standard  Wiener process. More precisely, the product 
Wiener measure $m\times m$ is rotation invariant in $C_{0}^2[0,T]$, 
see \cite[Lemmas 1 and 2]{CS76-2}. As discussed in \cite{CS76-2}, those 
rotation invariant properties of $m\times m$ were concretely realized by
Bearman \cite{B52}.

\par
Recently in  \cite{CSC12}, the authors used other rotation form of 
Wiener measure $m$  to defined a multiple  \emph{analytic FFT
associated with nonstationary Gaussian processes $\mathcal Z_h$}($\mathcal Z_h$-FFT) 
on $C_{0}[0,T]$. The rotation form used in \cite{CSC12} is a generalization of Bearman's  
celebrated result  and is intended to interpret   behaviors of nonstationary 
Gaussian processes on $C_0[0,T]$.  The authors also investigated various relationships 
which exist between the multiple FFT  and the 
corresponding CP associated with  nonstationary Gaussian 
processes on $C_{0}[0,T]$.

\par 
In this paper, motivated by the results in  \cite{CCKSY05,HPS95,HPS96,HPS97-1,HPS97-2}, 
we shall study the relationship between fundamental relationships between analytic 
$\mathcal Z_h$-FFTs and convolution type operations(CTO) on Wiener space.
The paper is organized as follows. In Section \ref{sec:pre} we briefly recall 
well-known results in Gaussian processes on Wiener space and give the concepts
of the $\mathcal Z_h$-FFT and the convolution type operation of functionals 
on Wiener space. In Section \ref{sec:E-space}, we emphasize the main purpose of 
this paper via specific examples. To do this we introduce the partially exponential 
type functionals  on $C_0[0,T]$. In Section \ref{sec:rotation}, as preliminary 
results, we investigate rotation properties of the Gaussian processes on Wiener space.
In Section  \ref{sec:FFT+COP}, we investigate  fundamental relationships between analytic 
$\mathcal Z_h$-FFTs and CTOs. Finally in Section \ref{sec:example}, we give several 
examples to apply our  assertions in this paper.

\setcounter{equation}{0}
\section{Preliminaries}\label{sec:pre}


For each $v\in L_2[0,T]$  and $x\in C_{0}[0,T]$,  we let 
$\langle{v,x}\rangle=\int_0^T v(t)d x(t)$
denote the Paley--Wiener--Zygmund  stochastic integral \cite{JS81,PWZ33,PS88}.
For any $h\in L_2[0,T]$ with $\|h\|_2>0$, let  $\mathcal{Z}_h$ be the stochastic  
process \cite{CPS93,HPS97-2,PS91,PS01} on $C_0[0,T]\times[0,T]$ given by 
\begin{equation}\label{eq:g-process}
\mathcal{Z}_h(x,t)
=\int_0^t h(s) d x(s)
=\langle{h\chi_{[0,t]},x}\rangle.
\end{equation}
Of course if $h(t)\equiv 1$ on $[0,T]$, 
then $\mathcal{Z}_h (x,t)=x(t)$ is an ordinary Wiener process.
It is known that the PWZ stochastic integral  $\langle{v,x}\rangle$ is Gaussian with 
mean zero and variance $\|v\|_2^2$, where $\|\cdot\|_2$ denotes the $L_2[0,T]$-norm.
Throughout this paper, we denote the Wiener integral of a Wiener 
measurable functional $F$ by 
\[
E[F]\equiv E_x[F(x)]=\int_{C_0[0,T]} F(x) dm(x).
\]

\par
Let  $\beta_h(t)=\int_0^th^2(u)du$.
It is easy to see that $\mathcal{Z}_h$ is a 
Gaussian process with mean zero and covariance 
function
\[
E_x[\mathcal{Z}_h(x,s)\mathcal{Z}_h(x,t)]
=\int_0^{\min\{s,t\}}h^2(u) du
=\beta({\min \{s,t\}}).
\]
In addition, $\mathcal{Z}_h (\cdot, t)$ is 
stochastically  continuous in $t$ on $[0,T]$ and   
for any  $h_1,h_2 \in  L_2[0,T]$,
\begin{equation}\label{eq:covh1h2}
E_x[\mathcal{Z}_{h_1}(x,s)\mathcal{Z}_{h_2}(x,t)]
=\int_{0}^{\min\{s,t\}}h_1(u)h_2(u) d u.
\end{equation}
Furthermore, if $h$ is of bounded variation on $[0,T]$,
then $\mathcal Z_h(x, t)$ is continuous in $t$ for all $x\in C_0[0,T]$,
i.e., $\mathcal Z_h(\cdot, \cdot)$ is a continuous process on $C_0[0,T]\times [0,T]$.
Thus throughout the remainder of this paper  we require $h$ to be in $BV[0,T]$,
the space of functions of bounded variation on $[0,T]$, for each process $\mathcal Z_h$.

It is known   \cite{CPS93} that for $v\in L_2[0,T]$ and $h\in L_{\infty}[0,T]$,
\begin{equation} \label{eq:Z-bsic-p}
\langle{v, \mathcal Z_h (x,\cdot)}\rangle =\langle{vh,x}\rangle
\end{equation}
for s-a.e. $x\in C_0[0,T]$.

 \par
Using \eqref{eq:Z-bsic-p} and the change of variable formula,
one can establish the integration formula on $C_0[0,T]$:
\begin{equation}\label{eq:int-formula-exponential}
\int_{C_{0}[0,T]} e^{\langle{v, \rho x}\rangle} dm(x)
=\exp\bigg\{ \frac{\rho^2}{2}\|v\|_2^2 \bigg\}
\end{equation}
for every $v\in L_2[0,T]$  and  $\rho\in \mathbb R\setminus\{0\}$.


\par
Throughout this paper, we will assume that
each functional $F$ (or $G$) we consider satisfies 
the 
conditions:
\begin{equation}\label{eq:s-condition01}
F:C_0[0,T] \to \mathbb C \,\,
\hbox{ is s.i.m. and s-a.e. defined},
\end{equation}
and for all $h\in L_2[0,T]$,
\begin{equation}\label{eq:s-condition02}
E_x\big[ \big|F\big(\rho \mathcal Z_h(x,\cdot)\big)\big|\big]<+\infty 
 \,\,\hbox{ for each }\, \rho>0.
\end{equation}

\par 
Throughout this paper, let $\mathbb C$, $\mathbb C_+$ and 
$\mathbb{\widetilde C}_+$ denote the set of complex numbers,  
complex numbers with positive real part, and  nonzero complex 
numbers  with nonnegative real part, respectively.
For each $\lambda \in \mathbb{\widetilde C}_+$, 
$\lambda^{1/2}$ denotes the principal square root of $\lambda$; i.e., 
$\lambda^{1/2}$ is always chosen to have positive real part, so that  
$\lambda^{-1/2}=(\lambda^{-1})^{1/2}$ is  in $\mathbb C_+$ for 
all $\lambda \in \widetilde{\mathbb C}_+$.

\begin{definition} \label{def:f-int}
Let $F$ satisfy conditions \eqref{eq:s-condition01} 
and \eqref{eq:s-condition02} above. Let $\mathcal Z_h$ be the Gaussian 
process given by \eqref{eq:g-process} and for $\lambda>0$, let
$J(h;\lambda)=E_{x}[F(\lambda^{-1/2}\mathcal Z_h(x,\cdot))]$. 
If there exists a function $J^* (h;\lambda)$ analytic on $\mathbb C_+$  
such that  $J^*(h;\lambda)=J(h;\lambda)$ 
for all $\lambda > 0$, then  $J^* (h;\lambda)$ is 
defined to be the analytic $\mathcal Z_h$-Wiener integral
(namely, the generalized analytic  Wiener integral with respect to 
the Gaussian paths $\mathcal{Z}_h(x,\cdot)$) 
of $F$  over $C_0[0,T]$ with parameter $\lambda$. In this case, 
for  $\lambda \in \mathbb C_+$ we write
\[
E_{x}^{\text{\rm an}w_{\lambda}} [F(\mathcal Z_h(x,\cdot))]
= J^*(h;\lambda).
\]
Let $q\ne 0$ be a real number and 
let $F$ be a functional such that the analytic $\mathcal Z_h$-Wiener integral
$E_{x}^{\text{\rm an}w_{\lambda}} [F(\mathcal Z_h(x,\cdot))]$ 
exists for all $\lambda \in \mathbb C_+$.
If the following limit exists, we call it 
the analytic $\mathcal Z_h$-Feynman integral (namely, the generalized analytic  Feynman 
integral with respect to the Gaussian paths $\mathcal{Z}_h(x,\cdot)$) 
of $F$ with parameter $q$ and we write
\begin{equation}\label{eq:F-int}
E_{x}^{\text{\rm an}f_{q}} [F(\mathcal Z_h(x,\cdot))]
=\lim_{\begin{subarray}{1} \lambda\to -iq \\  \lambda\in \mathbb C_+\end{subarray}}
E_{x}^{\text{\rm an}w_{\lambda}} [F(\mathcal Z_h(x,\cdot))].
\end{equation}
\end{definition}

\par 
Note that if $h\equiv 1$ on $[0,T]$, then these  definitions agree with 
the previous definitions of the  analytic Wiener integral and the analytic 
Feynman integral  \cite{CS76-1,CS80,HPS95,HPS96,HPS97-1,JS79-I,JS79-II,PS98,PSS98}.

\par 
Next we  state the definition of the generalized analytic FFT.

\begin{definition} \label{def:fft}
For $k\in L_2[0,T]$, $\lambda \in \mathbb{C}_+$, and $y \in C_{0}[0,T]$, 
let
\begin{equation}\label{eq:anal-T}
T_{\lambda,k}(F)(y)
= E_{x}^{\text{\rm an}w_{\lambda}} [F(y+\mathcal Z_k(x,\cdot))].
\end{equation}
We define the   $L_1$ analytic $\mathcal Z_k$-FFT (namely,
the analytic  FFT with respect to the Gaussian paths $\mathcal Z_k$), 
$T_{q,k}^{(1)}(F)$ of $F$, 
by the formula $(\lambda\in \mathbb C_+)$
\[
T_{q,k}^{(1)}(F)(y)=\lim_{\lambda\to-iq}
T_{\lambda,k} (F)(y),				 
\]
for s-a.e. $y \in C_0[0,T]$ whenever this limit exists.
\end{definition}

\par
We note that   if $T_{q,k}^{(1)}(F)$  exists  and if $F\approx G$, 
then $T_{q,k}^{(1)}(G)$ exists   and  $T_{q,k}^{(1)}(G)\approx T_{q,k}^{(1)}(F)$.
One can see that for each $k\in L_2[0,T]$, $T_{q,k}^{(1)}(F)\approx T_{q,-k}^{(1)}(F)$ 
since 
\[
\int_{C_0[0,T]}F(x)dm(x)=\int_{C_0[0,T]}F(-x)dm(x).
\]

By Definitions \ref{def:f-int} and \ref{def:fft}, it is easy to see that
for a nonzero real number $q$,
\begin{equation}\label{eq:fft+F-int}
T_{q,k}^{(1)}(F)(y) 
=E_{h}^{\text{\rm an}f_{q}} [F(y+\mathcal Z_k(x,\cdot))] 
\end{equation}
for s-a.e. $y\in C_0[0,T]$ if both sides exist.

\par
Next we give the definition of the CTO.

\begin{definition}\label{def:cp}
Let $F$ and $G$ be   functionals on $C_{0}[0,T]$.
For $\lambda \in \widetilde{\mathbb C}_+$, $g_1,g_2\in BV[0,T]$
and  $h_1, h_2 \in L_2[0,T]$, 
we define their  CTO  with respect to 
$\{\mathcal{Z}_{g_1},\mathcal{Z}_{g_2},\mathcal{Z}_{h_1},\mathcal{Z}_{h_2}\}$ 
 (if it exists) by
\begin{equation}\label{eq:gcp-Z}
\begin{aligned}
& (F*G)_{\lambda}^{(g_1,g_2;h_1,h_2)}(y)\\
&=
\begin{cases}
E_x^{\text{\rm  an}w_{\lambda}} [
F (
\mathcal Z_{g_1} (y,\cdot)+ \mathcal Z_{h_1} (x,\cdot) )
G (
\mathcal Z_{g_2} (y,\cdot)+ \mathcal Z_{h_2} (x,\cdot) ) ],   \quad   \lambda \in \mathbb C_+ \\
E_x^{\text{\rm  an}f_{q}} [
F (
\mathcal Z_{g_1} (y,\cdot)+ \mathcal Z_{h_1} (x,\cdot) )
G (
\mathcal Z_{g_2} (y,\cdot)+ \mathcal Z_{h_2} (x,\cdot) )],\\
 \qquad\qquad\qquad\qquad\qquad\qquad\qquad\qquad\quad \quad \, 
          \lambda=-iq,\,\, q\in \mathbb R, \,\,q\ne 0.
\end{cases}
\end{aligned}
\end{equation}
When $\lambda =-iq$,   we will denote
$(F*G)_{\lambda}^{(g_1,g_2;h_1,h_2)}$
by $(F*G)_{q}^{(g_1,g_2;h_1,h_2)}$.
\end{definition}

\begin{remark}
(i)  Given  a function  $h$ in $L_2[0,T] \setminus\{0\}$, letting 
$h_1=-h_2=h/\sqrt{2}$ and $g_1=g_2\equiv1/\sqrt{2}$,  
equation \eqref{old-cp-g} reduces  the convolution
product   studied in \cite{CCKSY05,CSC12,HPS97-2,PS01}:
\begin{equation*}\label{old-cp-g}
(F*G)_{q}^{ (g_1,g_2;h_1,h_2) }(y)
=E_x^{\text{\rm  an}f_{q}}\bigg[
F\bigg(\frac{y+ \mathcal Z_{k} (x,\cdot)}{\sqrt2}\bigg)
G\bigg(\frac{y- \mathcal Z_{k} (x,\cdot)}{\sqrt2}\bigg)\bigg]. 
\end{equation*}

(ii) Choosing $h_1=-h_2=g_1=g_2 \equiv 1/\sqrt{2}$,  we have
the convolution product studied in \cite{HPS95,HPS96,HPS97-1}:
\begin{equation*}\label{old-cp-w}
(F*G)_{q}^{ (g_1,g_2;h_1,h_2) }(y)
=E_x^{\text{\rm  an}f_{q}}
\bigg[F\bigg(\frac{y+ x}{\sqrt2}\bigg)
G\bigg(\frac{y- x}{\sqrt2}\bigg)\bigg].
\end{equation*}

(iii) Choosing $h_1=h_2=g_1=-g_2 = 1/\sqrt{2}$ and $\lambda=1$,  we have
the convolution product studied in \cite{Yeh65}:
\begin{equation*}\label{old-cp-yeh}
(F*G)^{ (g_1,g_2;h_1,h_2) }(y)
=E_x\bigg[F\bigg(\frac{y+ x}{\sqrt2}\bigg)
G\bigg(\frac{-y+x}{\sqrt2}\bigg)\bigg].
\end{equation*}
\end{remark}

\setcounter{equation}{0}
\section{Observation on the class $\mathcal E (C_{0}[0,T])$}\label{sec:E-space}

\par
Let $\mathcal{E}$ be the class of all functionals having the form
\begin{equation}\label{eq:Psi-weight-s}
\Psi_u (x) =e^{\langle{u,x}\rangle } \,\,\mbox{ for  a.e. } x\in C_{0}[0,T]
\end{equation}
with  $u \in L_2[0,T]$, and given  $q\in\mathbb R\setminus \{0\}$,
$v\in L_2[0,T]$, 
and $k\in BV[0,T]$, let $\mathcal E_{q,v,k}$  be the class of all 
functionals having the form 
\begin{equation}\label{eq:Psi-weight-sh}
\Psi_{u}^{q,v,k} (x) 
=\Psi_u (x)\exp\bigg\{\frac{i}{2q}\|vk\|_{2}^2\bigg\} \,\,\mbox{ for  a.e. } x\in C_{0}[0,T],
\end{equation}  
where $\Psi_u $ is given by equation \eqref{eq:Psi-weight-s}. 
We note that $\Psi_{u}$ and $\Psi_{u}^{q,v,k}$ are scale invariant measurable.

The functionals given by equation   
\eqref{eq:Psi-weight-sh} and linear combinations (with complex coefficients) 
of the  $\Psi_{u}^{q,v,k} $'s are called the partially  exponential type   
functionals on $C_0[0,T]$. The functionals given by  \eqref{eq:Psi-weight-s} 
are also partially  exponential type   functionals because 
$\Psi_{u}^{q,v,0} (x)=\Psi_u (x)$ for s-a.e. $x\in C_{0}[0,T]$.

For each $(q,v,k)\in (\mathbb R\setminus\{0\})\times L_2[0,T]\times BV[0,T]$, the class  $\mathcal E_{q,v,k}$ is
dense in $L_2(C_{0}[0,T])$. Furthermore,   ${\text{\rm Span}}\mathcal E_{q,v,k}$,
the linear manifold generated by $\mathcal  E_{q,v,k}$ in $L_2(C_0[0,T])$, is 
closed under the ordinary multiplication
because  
\[
\Psi_{u_1}^{q,v,k} (x) \Psi_{u_2}^{q,v,k} (x) 
=\alpha \exp\bigg\{\langle{u_1+u_2,x}\rangle +\frac{i}{2q}\|vk\|_{2}^2\bigg\} 
=\alpha\Psi_{u_1+u_2}^{q,v,k} (x) 
\] 
for s-a.e. $y\in C_0[0,T]$, where 
with the complex coefficient $\alpha$  given by
$\exp\{(i/2q)\|vk\|_{2}^2\}$.
Thus the class  ${\text{\rm Span}}\mathcal E_{q,v,k}$ 
is a commutative algebra over the complex field $\mathbb C$.

In fact, using the fact that
\[
\Psi_{u_1}^{q_1,v_1,k_1} (x) \Psi_{u_2}^{q_2,v_2,k_2} (x) 
=\beta \exp\{\langle{u_1+u_2,x}\rangle\}
=\beta \Psi_{u_1+u_2} (x) 
\] 
with
\[
\beta
=\exp\bigg\{\frac{i}{2q_1}\|v_1k_1\|_{2}^2 +\frac{i}{2q_2}\|v_2k_2\|_{2}^2\bigg\},
\]
one can see that
\[
{\text{\rm Span}}\bigg( 
\bigcup_{\begin{subarray}{1} q\in \mathbb R  \\  
 v\in L_2[0,T]\\
 k\in BV[0,T]    \end{subarray}}  \mathcal E_{q,v,k}\bigg)
={\text{\rm Span}}\mathcal E.
\] 
We denote the set of all partially  exponential type 
functionals on $C_{0}[0,T]$ by $\mathcal E (C_{0}[0,T])$, i.e.
$\mathcal E (C_{0}[0,T])={\text{\rm Span}}\mathcal E$.

\par
First, using  \eqref{eq:F-int} with $F$  replaced 
with $\Psi_{u}$,  \eqref{eq:Z-bsic-p}, 
\eqref{eq:int-formula-exponential}, it follows that for all 
$k\in BV[0,T]\setminus\{0\}$,
\begin{equation}\label{evalu-gfft-Phi-pre}
E_{x}^{\text{\rm an}f_{q}} [\Psi_u(\mathcal Z_k(x,\cdot))]
=\exp\bigg\{\frac{i}{2q}\|uk\|_{2}^2 \bigg\}. 
\end{equation}
Thus, using equations \eqref{eq:fft+F-int} and \eqref{evalu-gfft-Phi-pre}, 
we see that  the $L_1$ analytic $\mathcal Z_k$-FFT, $T_{q,k}^{(1)}$,
exists for all $q\in \mathbb R\setminus\{0\}$, and is given by  
\begin{equation}\label{evalu-gfft-Phi}
T_{q,k}^{(1)}(\Psi_u )(y)
=\Psi_u  (y)E_{x}^{\text{\rm an}f_{q}} [\Psi_u(\mathcal Z_k(x,\cdot))]
=\Psi_{u}^{q,u,k} (y) 
\end{equation}
for s-a.e. $y\in C_{0}[0,T]$. From equation \eqref{evalu-gfft-Phi}, 
we also see that  
$T_{q,k}^{(1)}: \mathcal E (C_{0}[0,T]) \to \mathcal E (C_{0}[0,T])$ 
is well defined.

\par
Next, using \eqref{eq:gcp-Z} with $F$ and $G$ replaced with 
$\Psi_{u} $ and $\Psi_{v} $, it follows that
for all real $q\in \mathbb R\setminus\{0\}$ and $g_1,g_2,h_3,h_4\in L_2[0,T]$,
the CTO of $\Psi_{u}$ and $\Psi_{v}$, $(\Psi_{u} *\Psi_{v} )_q^{(g_1,g_2;h_3,h_4)}$, 
exists  and is given by  
\begin{equation}\label{evalu-cto-Phis} 
(\Psi_{u} *\Psi_{v} )_q^{(g_1,g_2;h_1,h_2)}(y)
= \exp\bigg\{\langle{ug_1+vg_2,y}\rangle
+\frac{i}{2q}\|uh_1+vh_2\|_2^2\bigg\}
 \end{equation}
for s-a.e. $y\in C_{0}[0,T]$. 
Also, the functional  $(\Psi_{u} *\Psi_{v} )_q^{(g_1,g_2;h_3,h_4)}\equiv \Psi_{ug_1+vg_2}^{q,uh_1+vh_2,1} $ 
is an element of $\mathcal E(C_0[0,T])$.

\par
Using \eqref{evalu-cto-Phis}  and applying the techniques similar to those used in \eqref{evalu-gfft-Phi},
one can see that for s-a.e. $y\in C_0[0,T]$,
\[
\begin{aligned}
&T_{q, k}((\Phi_u*\Phi_v)_{q}^{(g_1,g_2;h_1,h_2)})(y)\\
&=\exp\bigg\{\langle{ug_1+vg_2,y}\rangle
+\frac{i}{2q}\|uh_1+vh_2\|_2^2
+\frac{i}{2q}\|(ug_1+vg_2)k\|_2^2\bigg\}\\
&=\exp\bigg\{\langle{ug_1+vg_2,y}\rangle
+\frac{i}{q}\int_0^Tu(t)v(t)\big(h_1(t)h_2(t)+g_1(t)g_2(t)k^2(t)\big)dt\\
&\qquad
+\frac{i}{2q}\int_0^Tu^2(t)\big(h_1^2(t)+g_1^2(t)k^2(t)\big)dt
+\frac{i}{2q}\int_0^Tv^2(t)\big(h_2^2(t)+g_2^2(t)k^2(t)\big)dt\bigg\}.
\end{aligned}
\]
In order to obtain an equation similar to \eqref{eq:intro},
one may put the condition that
\begin{equation}\label{condition-ob01}
h_1(t)h_2(t)+g_1(t)g_2(t)k^2(t)=0  \,\,\,\,\,m_L\mbox{-a.e. on }\,  [0,T].
\end{equation}
Then we can expect the following equation: 
\begin{equation}\label{eq:observation-A}
T_{q, k}((\Phi_u*\Phi_v)_{q}^{(g_1,g_2;h_1,h_2)})(y)
=T_{q,\mathbf{s}_1}(\Phi_u)(\mathcal Z_{g_1}(y,\cdot))
 T_{q,\mathbf{s}_2}(\Phi_v)(\mathcal Z_{g_2}(y,\cdot)) 
\end{equation}
for s-a.e. $y\in C_0[0,T]$, where $\mathbf{s}_i(i=1,2)$  is the function of bounded 
variation on $[0,T]$ such that $\mathbf{s}_i^2(t)=g_i^2(t)k^2(t)+h_i^2(t)$ for $m_L$-a.e. on  $[0,T]$.

On the other hand, using \eqref{evalu-gfft-Phi} and \eqref{eq:gcp-Z}, we 
also obtain that  for s-a.e. $y\in C_0[0,T]$,
\[
\begin{aligned}
&\big(T_{q, k_1}(F)*T_{q, k_2}(G)\big)_{q}^{(g_1,g_2;h_3,h_4)} (y)\\
&=\exp\bigg\{\langle{ug_1+vg_2,y}\rangle
+\frac{i}{q}\int_0^Tu(t)v(t) h_3(t)h_4(t)dt\\
&\qquad\quad 
+\frac{i}{2q}\int_0^Tu^2(t)\big(h_3^2(t)+k_1^2(t)\big)
+\frac{i}{2q}\int_0^Tv^2(t)\big(h_4^2(t)+k_2^2(t)\big)\bigg\},
\end{aligned}
\]
and under the condition 
\begin{equation}\label{condition-ob02}
m_L(\mbox{supp}(h_3)\cap \mbox{supp}(h_4))=0.    
\end{equation}   
it follows that
\begin{equation}\label{eq:observation-B}
\begin{aligned}
 (T_{q, k_1}(\Phi_u)*T_{q, k_2}(\Phi_v) )_{q}^{(g_1,g_2;h_3,h_4)} (y) 
 =T_{q,\mathbf{s}_3}(\Phi_u) (\mathcal Z_{g_1} (y,\cdot) )
  T_{q,\mathbf{s}_4}(\Phi_v) (\mathcal Z_{g_2} (y,\cdot) ) 
\end{aligned}
\end{equation}
for s-a.e. $y\in C_0[0,T]$, where $\mathbf{s}_3$ and $\mathbf{s}_4$ are the functions 
of bounded  variation on $[0,T]$ such that 
$\mathbf{s}_3^2(t)=h_3^2(t)+k_1^2(t)$ and
$\mathbf{s}_4^2(t)=h_4^2(t)+k_2^2(t)$.

In Section  \ref{sec:FFT+COP} below, we establish the 
relationships appearing in \eqref{eq:observation-A} and \eqref{eq:observation-B}
for  general functionals $F$ and $G$ on Wiener space.
Equations \eqref{condition-ob01} and \eqref{condition-ob02} play key roles in 
our main theorems (see Theorems \ref{thm:p-main} and \ref{thm:p-main-II} 
below) in this paper.

\setcounter{equation}{0}
\section{Rotation properties of Gaussian processes}\label{sec:rotation}

The essential  purpose of this  section  is to establish   rotation 
properties of Gaussian processes on the product Wiener  spaces $C_{0}^2[0,T]$ and $C_{0}^3[0,T]$. 
For $h_1,h_2\in BV[0,T]$ with $\|h_j\|_{2}>0$, $j\in\{1,2\}$,
let $\mathcal Z_{h_1}$ and $\mathcal Z_{h_2}$ be the Gaussian   processes given
by \eqref{eq:g-process} with $h$ replaced with $h_1$ and $h_2$ respectively.
Then the process
\[
\mathfrak Z_{h_1,h_2}: C_{0}[0,T]\times C_{0}[0,T]\times [0,T]\to \mathbb R
\]
given by
\begin{equation*}\label{eq:frak-Z} 
\mathfrak Z_{h_1,h_2}(x_1,x_2,t)
= \mathcal Z_{h_1}(x_1,t)+ \mathcal Z_{h_2}(x_2,t)
\end{equation*}
is also a Gaussian process with mean zero
and covariance function  
\[
\begin{aligned}
&E_{x_1}\big[ E_{x_2}\big[
\mathfrak Z_{h_1,h_2}(x_1,x_2,s)\mathfrak Z_{h_1,h_2}(x_1,x_2,t) \big]\big]\\
&=\beta_{h_1}(\min\{s,t\})+\beta_{h_2}(\min\{s,t\})\\
&=\mathfrak{v}_{h_1,h_2}(\min\{s,t\}).\\
\end{aligned}
\]

 On the other hand, let $h_1$ and $h_2$ be elements of $BV[0,T]$. Then there exists a
function  $\mathbf{s}\in BV[0,T]$ such that
\begin{equation}\label{eq:fn-rot}
 \mathbf{s}^2(t)=h_1^2(t)+h_2^2(t)
\end{equation}
for $m_L$-a.e. $t\in [0,T]$, where $m_L$ denotes the Lebesgue on $[0,T]$.
Note that  the function `$\mathbf{s}$' satisfying \eqref{eq:fn-rot} is not unique.
We will use the symbol  $\mathbf{s}(h_1,h_2)$ for the
functions `$\mathbf{s}$' that satisfy \eqref{eq:fn-rot} above.

Given $h_1,h_2\in BV[0,T]$, we consider the stochastic process $\mathcal Z_{\mathbf{s}(h_1,h_2)}$.
Then $\mathcal Z_{\mathbf{s}(h_1,h_2)}$ is a Gaussian process with mean zero
and covariance
\[
\begin{aligned}
&E_x\big[ \mathcal Z_{\mathbf{s}(h_1,h_2)}(x,s) 
 \mathcal Z_{\mathbf{s}(h_1,h_2)}(x,t) \big] \\
&=\int_0^{\min\{s,t\}} \mathbf{s}^2(h_1,h_2)(u)db(u)
=\int_0^{\min\{s,t\}} \big(h_1^2(u)+h_2^2(u)\big)db(u)\\
&=\beta_{h_1}(\min\{s,t\})+\beta_{h_2}(\min\{s,t\}) 
=\mathfrak v_{h_1,h_2}(\min\{s,t\}) .
\end{aligned}
\]

\par
From these facts, one can see that  $\mathfrak Z_{h_1,h_2}$ and $ \mathcal Z_{\mathbf{s}(h_1,h_2)}$
have the same distribution and that for any random variable $F$ on $C_{0}[0,T]$,
\begin{equation}\label{eq:1-Dim}
E_{x_1}\Big[E_{x_2}\Big[
F\Big(\mathcal Z_{h_1}(x_1,\cdot)+ \mathcal Z_{h_2}(x_2,\cdot)\Big)\Big]\Big] 
\stackrel{*}{=}
E_x\big[F\big(\mathcal Z_{\mathbf{s}(h_1,h_2)}(x,\cdot)\big)\big],
\end{equation}
where by $\stackrel{*}{=}$ we mean  that if either side exists,
both sides exist  and equality holds.

\subsection{A rotation property of Gaussian processes on $C_0^2[0,T]$}
The following  lemma  will be very useful in our main theorem.

\begin{lemma}\label{lemm:cov}
Given functions  $h_1$,  $h_2$, $h_3$ and $h_4$ in $L_2[0,T]\setminus\{0\}$, 
let the two stochastic processes $\mathfrak Z_{h_1,h_2}$ and 
$\mathfrak Z_{h_3,h_4}$ on $C_{0}^2[0,T]\times[0,T]$ be given by
\begin{equation}\label{eq:frak-Z0a1}
\mathfrak Z_{h_1,h_2}(x_1,x_2,t)
=\mathcal Z_{h_1}(x_1,t)+ \mathcal Z_{h_2}(x_2,t)
\end{equation}
and
\begin{equation}\label{eq:frak-Z0a2}
\mathfrak Z_{h_3,h_4}(x_1,x_2,t)
=\mathcal Z_{h_3}(x_1,t)+ \mathcal Z_{h_4}(x_2,t),
\end{equation}
respectively. Then the following assertions are equivalent.
\begin{itemize}
\item[(i)] $\mathfrak Z_{h_1,h_2}$ and $\mathfrak Z_{h_3,h_4}$ 
            are independent processes,
\item[(ii)] $h_1  h_3+ h_2  h_4=0$.
\end{itemize}
\end{lemma}
\begin{proof}
Since the processes $\mathfrak Z_{h_1,h_2}$ and $\mathfrak Z_{h_3,h_4}$ are 
Gaussian, we know that $\mathfrak Z_{h_1,h_2}$ and $\mathfrak Z_{h_3,h_4}$ 
are independent if and only if
\[
E_{x_2}[E_{x_1}[ \mathfrak Z_{h_1,h_2}(x_1,x_2,s)
\mathfrak Z_{h_3,h_4}(x_1,x_2,t)]]
= 0
\]
for all $s,t\in[0,T]$. But using the Fubini theorem and equation \eqref{eq:covh1h2}, 
we have
\[
\begin{aligned}
&E_{x_2}[E_{x_1}[\mathfrak Z_{h_1,h_2}(x_1,x_2,s)
\mathfrak Z_{h_3,h_4}(x_1,x_2,t)]]\\
&=E_{x_2}[E_{x_1}[
\big(\mathcal Z_{h_1}(x_1,s)\mathcal Z_{h_3}(x_1,t)
    +\mathcal Z_{h_1}(x_1,s)\mathcal Z_{h_4}(x_2,t)\\
&\qquad\qquad 
    +\mathcal Z_{h_2}(x_2,s)\mathcal Z_{h_3}(x_1,t)
    +\mathcal Z_{h_2}(x_2,s)\mathcal Z_{h_4}(x_2,t)\big) ]]\\
&=\int_0^{\min\{s,t\}}h_1(u)h_3(u)du
+\int_0^{\min\{s,t\}}h_2(u)h_4(u)du\\
&=\int_0^{\min\{s,t\}}\big(h_1(u)h_3(u)+h_2(u)h_4(u)\big)du.
\end{aligned}
\]
From this we can obtain the desired result.
\end{proof}

\begin{theorem}\label{theorem:2-Dim-1234}
Let $h_1$, $h_2$, $h_3$ and $h_4$ be nonzero elements of $BV[0,T]$ with 
$h_1 h_3+ h_2h_4=0$, and let $\mathbf{F}:C_{0}^2[0,T]\to\mathbb C$ be 
a $m\times m$-integrable functional. Then
\[ 
\begin{aligned}
&E_{x_1}\big[E_{x_2}\big[\mathbf{F}
\big(\mathcal Z_{h_1}(x_1,\cdot)+\mathcal Z_{h_2}(x_2,\cdot),
 \mathcal Z_{h_3}(x_1,\cdot)+\mathcal Z_{h_4}(x_2,\cdot)\big)\big]\big]\\
&=E_{y}\big[E_{x}\big[\mathbf{F}
\big(\mathcal Z_{\mathbf{s}(h_1,h_2)}(x,\cdot),
\mathcal Z_{\mathbf{s}(h_3,h_4)} (y,\cdot)\big)\big]\big].
\end{aligned}
\]
\end{theorem}
\begin{proof}
Let the processes $\mathfrak Z_{h_1,h_2}, \mathfrak Z_{h_3,h_4}:
C_{0}^2[0,T]\times[0,T]\to\mathbb R$ be given by equation 
\eqref{eq:frak-Z0a1} and \eqref{eq:frak-Z0a2} respectively. Since $h_i$'s 
are functions of bounded variation, for all $(x_1,x_2)\in C_{0}^2[0,T]$ 
the sample paths $\mathfrak Z_{h_1,h_2}(x_1,x_2,\cdot)$ and 
$\mathfrak Z_{h_3,h_4}(x_1,x_2,\cdot)$ of the processes are continuous 
functions on $[0,T]$. Let $X_{h_1,h_2}$ and  $X_{h_3,h_4}$ be measurable 
transforms from $C_{0}^2[0,T]$ into $C_{0}^2[0,T]$  given by
\[
X_{h_1,h_2}(x_1,x_2)=\mathfrak Z_{h_1,h_2}(x_1,x_2,\cdot)
\]
and
\[
X_{h_3,h_4}(x_1,x_2)=\mathfrak Z_{h_3,h_4}(x_1,x_2,\cdot)
\]
respectively. Also let $P\equiv X_{h_1,h_2}(C_{0}^2[0,T])$ and 
$Q \equiv X_{h_3,h_4}(C_{0}^2[0,T])$ be the image spaces of the measurable 
transforms $X_{h_1,h_2}$ and $X_{h_3,h_4}$ respectively. For simplicity, 
let $m^2$ denote the product Wiener measure $m\times m$ on $C_0^2[0,T]$.

\par
By Lemma \ref{lemm:cov}, we see that $\mathfrak Z_{h_1,h_2}$ 
and $\mathfrak Z_{h_3,h_4}$ are independent processes on $C_{0}^2[0,T]$ 
and so $X_{h_1,h_2}$ and $X_{h_3,h_4}$ are independent measurable transforms. 
Thus, by the change of variables formula, the Fubini theorem  and 
 \eqref{eq:1-Dim}, it follows that
\[ 
\begin{aligned}
&E_{x_2}[E_{x_1}[\mathbf{F}
\big(\mathfrak Z_{h_1,h_2}(x_1,x_2,\cdot), 
     \mathfrak Z_{h_3,h_4}(x_1,x_2,\cdot)\big)]]\\
&=\int_{C_{0}^2[0,T]}\mathbf{F}
\big(X_{h_1,h_2}(x_1,x_2),X_{h_3,h_4}(x_1,x_2)\big)dm^2(x_1,x_2)\\
&=\int_{P\times Q}\mathbf{F}(z_1,z_2)  
d\big[(m^2\circ X_{h_1,h_2}^{-1})\times(m^2\circ X_{h_3,h_4}^{-1})\big](z_1,z_2)\\
&=\int_{Q}\bigg[\int_{P}\mathbf{F}(z_1,z_2)  
d(m^2\circ X_{h_1,h_2}^{-1})(z_1)\bigg]d(m^2\circ X_{h_3,h_4}^{-1})(z_2)\\
&=\int_{Q}\bigg[\int_{C_{0}^2[0,T]}\mathbf{F}
\big(X_{h_1,h_2}(x_1,x_2),z_2\big)dm^2(x_1,x_2)\bigg]
d(m^2\circ X_{h_3,h_4}^{-1})(z_2)\\
&=\int_{Q}\bigg[\int_{C_{0}^2[0,T]}\mathbf{F}
\big(\mathcal Z_{h_1}(x_1,\cdot)+\mathcal Z_{h_2}(x_2,\cdot),z_2\big)
dm^2(x_1,x_2)\bigg]d(m^2\circ X_{h_3,h_4}^{-1})(z_2)\\
&=\int_{Q}\bigg[\int_{C_{0}[0,T]}\mathbf{F}
\big(\mathcal Z_{\mathbf{s}(h_1,h_2)}(x,\cdot),z_2\big)dm(x)\bigg]
d(m^2\circ X_{h_3,h_4}^{-1})(z_2)\\
&=\int_{C_{0}[0,T]} \bigg[\int_{Q}\mathbf{F}
\big(\mathcal Z_{\mathbf{s}(h_1,h_2)}(x,\cdot),z_2\big) 
d(m^2\circ X_{h_3,h_4}^{-1})(z_2)\bigg]dm(x)\\
&=\int_{C_{0}[0,T]} \bigg[\int_{C_{0}^2[0,T]}\mathbf{F}
\big(\mathcal Z_{\mathbf{s}(h_1,h_2)}(x,\cdot),X_{h_3,h_4}(x_1,x_2)\big)
dm^2(x_1,x_2)\bigg]dm(x)\\
&=\int_{C_{0}[0,T]} \bigg[\int_{C_{0}^2[0,T]}\mathbf{F}
\big(\mathcal Z_{\mathbf{s}(h_1,h_2)}(x,\cdot),  
\mathcal Z_{h_3}(x_1,\cdot) +\mathcal Z_{h_4}(x_2,\cdot)\big)
dm^2(x_1,x_2)\bigg]dm(x)\\
&=\int_{C_{0}[0,T]}\bigg[\int_{C_{0}[0,T]}\mathbf{F}
\big(\mathcal Z_{\mathbf{s}(h_1,h_2)}(x,\cdot),
\mathcal Z_{\mathbf{s}(h_3,h_4)}(y,\cdot)\big)
dm(y)\bigg]dm(x).
\end{aligned}
\]
Thus we obtain the desired result.
\end{proof}

The following corollaries  are very simple consequences 
of Theorem \ref{theorem:2-Dim-1234}.

\begin{corollary}\label{theorem:2-Dim}
Let $h_1$ and  $h_2$ be nonzero  elements of $BV[0,T]$ and 
let  $\mathbf{F}: C_{0}^2[0,T]\to\mathbb C$ be a  $m^2$-integrable 
functional. Then
\begin{equation}\label{eq:more-Bearman}
\begin{aligned}
&E_{x_2}[E_{x_1}[\mathbf{F}
\big(\mathcal Z_{h_1}(x_1,\cdot)-\mathcal Z_{h_2}(x_2,\cdot),
     \mathcal Z_{h_2}(x_1,\cdot)+\mathcal Z_{h_1}(x_2,\cdot)\big)]]\\
&=E_{y}[E_{x}[\mathbf{F}
\big(\mathcal Z_{\mathbf{s}(h_1, h_2)}(x,\cdot),
\mathcal Z_{\mathbf{s}(h_1, h_2)}(y,\cdot)\big)]].
\end{aligned}
\end{equation}
\end{corollary}

\begin{remark}
For any function $\theta(\cdot)$ of bounded variation, choosing 
$h_1(t)=\cos\theta(t)$ and $h_2(t)=\sin\theta(t)$ on $[0,T]$ in  
equation  \eqref{eq:more-Bearman} yields the main result in  \cite{B52}.
\end{remark}

\begin{corollary}\label{theorem:2-Dim12}
Let $h_1$,  $h_2$, $h_3$ and  $h_4$ be as 
in Theorem \ref{theorem:2-Dim-1234}.
Let $F$ and $G$ be $\mu$-integrable functionals.
Then
\begin{equation}\label{eq:R-pwz4-1234}
\begin{aligned}
&E_{x_1}\big[ E_{x_2}\big[
F\big(\mathcal Z_{h_1}(x_1,\cdot)+ \mathcal Z_{h_2}(x_2,\cdot)\big)
G\big(\mathcal Z_{h_3}(x_1,\cdot)+ \mathcal Z_{h_4}(x_2,\cdot)\big) \big]\big]\\
&=E_x\big[F\big(\mathcal Z_{\mathbf{s}(h_1,h_2)}(x,\cdot)\big)\big]
E_x\big[G\big(\mathcal Z_{\mathbf{s}(h_3,h_4)}(x,\cdot)\big)\big]
\end{aligned}
\end{equation}
\end{corollary}

\subsection{A rotation property of Gaussian processes on $C_0^3[0,T]$}

\begin{lemma}\label{lemm:cov}
Given functions  $h_1$,  $h_2$, $h_3$ and $h_4$ in 
$L_2[0,T]\setminus\{0\}$, 
let the two stochastic processes $\mathfrak Z_{h_1,h_2,0}$ 
and $\mathfrak Z_{h_3,0,h_4}$ on $C_{0}^{3}[0,T]\times[0,T]$
be given by
\begin{align}
\mathfrak W_{h_1,h_2,0}(x_1,x_2,x_3,t)
&=\mathcal Z_{h_1}(x_1,t)+ \mathcal Z_{h_2}(x_2,t)  \label{eq:frak-Z001}\\
\intertext{and} 
\mathfrak W_{h_3,0,h_4}(x_1,x_2,x_3,t)
&=\mathcal Z_{h_3}(x_1,t) + \mathcal Z_{h_4}(x_3,t) \label{eq:frak-Z002},
\end{align}
respectively. 
If $m_L (\text{\rm supp}(h_1)\cap  \text{\rm supp}(h_3))=0$, then
$\mathfrak W_{h_1,h_2,0}$ and $\mathfrak W_{h_3,0,h_4}$ are independent processes.
\end{lemma}

\begin{remark}\label{consistency}
By the consistency property, the processes $\mathfrak W_{h_1,h_2,0}$ 
and $\mathfrak W_{h_1,0,h_3}$  can be considered as processes on $C_0^2[0,T]\times[0,T]$.
\end{remark}

\begin{proof}[Proof of Lemma \ref{lemm:cov}]
Since the processes $\mathfrak W_{h_1,h_2,0}$ and $\mathfrak W_{h_3,h_4}$ are 
Gaussian, we know that $\mathfrak W_{h_1,h_2,0}$ and $\mathfrak W_{h_3,0,h_4}$ 
are independent if and only if
\[
E_{x_3}\big[ E_{x_2}\big[E_{x_1}\big[\mathfrak W_{h_1,h_2,0}(x_1,x_2,x_3,s)
\mathfrak W_{h_3,0,h_4}(x_1,x_2,x_3,t)\big]\big]\big]
= 0
\]
for all $s,t\in[0,T]$. But using the Fubini theorem 
and equation \eqref{eq:covh1h2}, we have
\[
\begin{aligned}
&E_{x_3}\big[ E_{x_2}\big[E_{x_1}\big[\mathfrak W_{h_1,h_2,0}(x_1,x_2,x_3,s)
\mathfrak W_{h_3,0,h_4}(x_1,x_2,x_3,t)\big]\big]\big]\\
&=E_{x_3}\Big[ E_{x_2}\Big[E_{x_1}\Big[
\Big(\mathcal Z_{h_1}(x_1,s)\mathcal Z_{h_3}(x_1,t)
    +\mathcal Z_{h_1}(x_1,s)\mathcal Z_{h_4}(x_3,t)\\
&\qquad\qquad\qquad
    +\mathcal Z_{h_2}(x_2,s)\mathcal Z_{h_3}(x_1,t)
    +\mathcal Z_{h_2}(x_2,s)\mathcal Z_{h_4}(x_3,t)\Big) \Big]\Big]\Big]\\
&=\int_0^{\min\{s,t\}}h_1(u)h_3(u)du
\end{aligned}
\]
From this we can obtain the desired result.
\end{proof}

\begin{theorem}\label{theorem:2-Dim-II}
Let $h_1$, $h_2$, $h_3$ and $h_4$ be nonzero elements 
of $BV[0,T]$ with 
\[
m_L (\text{\rm supp}(h_1)\cap  \text{\rm supp}(h_3))=0,
\] 
and let $\mathbf{F}:C_{0}^2[0,T]\to\mathbb C$ be 
a $m^2$-integrable functional. Then
\[ 
\begin{aligned}
&E_{x_3}\Big[ E_{x_2}\Big[E_{x_1}\Big[\mathbf{F}
\Big(\mathcal Z_{h_1}(x_1,\cdot)+\mathcal Z_{h_2}(x_2,\cdot),
     \mathcal Z_{h_3}(x_1,\cdot)+\mathcal Z_{h_4}(x_3,\cdot)\Big)\Big]\Big]\Big]\\
&=E_{y}\Big[ E_{x}\Big[\mathbf{F}\Big(\mathcal Z_{\mathbf{s}(h_1,h_2)}(x,\cdot),
\mathcal Z_{\mathbf{s}(h_3,h_4)} (y,\cdot)\Big)\Big]\Big].
\end{aligned}
\]
\end{theorem}
\begin{proof}
Let the processes $\mathfrak W_{h_1,h_2,0}, \mathfrak W_{h_3,0,h_4}:C_{0}^3[0,T]\times[0,T]\to\mathbb R$ 
be given by equation \eqref{eq:frak-Z001} and \eqref{eq:frak-Z002} respectively. Since $h_i$'s 
are functions of bounded variation, for all $(x_1,x_2,x_3)\in C_{0}^3[0,T]$ 
the sample paths $\mathfrak W_{h_1,h_2,0}(x_1,x_2,x_3,\cdot)$ and 
$\mathfrak W_{h_3,0,h_4}(x_1,x_2,x_3,\cdot)$ of the processes are continuous 
functions on $[0,T]$. Let $Y_{h_1,h_2,0}$ and  $Y_{h_3,0,h_4}$ be measurable 
transforms from $C_{0}^3[0,T]$ into $C_{0}^3[0,T]$  given by
\[
Y_{h_1,h_2,0}(x_1,x_2,x_3)=\mathfrak W_{h_1,h_2,0}(x_1,x_2,x_3,\cdot)
\]
and
\[
Y_{h_3,0,h_4}(x_1,x_2,x_3)=\mathfrak W_{h_3,0,h_4}(x_1,x_2,x_3,\cdot)
\]
respectively. Also let 
$M\equiv Y_{h_1,h_2,0}(C_{0}^3[0,T])$ and 
$N\equiv Y_{h_3,0,h_4}(C_{0}^3[0,T])$ be the 
image spaces of the measurable transforms 
$Y_{h_1,h_2,0}$ and $Y_{h_3,0,h_4}$ respectively.
For simplicity, let $m^3$ denote the product Wiener measure $m\times m\times m$
on $C_0^3[0,T]$.

\par
By Lemma \ref{lemm:cov}, we see that $\mathfrak W_{h_1,h_2,0}$ 
and $\mathfrak W_{h_3,0,h_4}$ are independent processes on $C_{0}^3[0,T]$ 
and so $Y_{h_1,h_2,0}$ and $Y_{h_3,0,h_4}$ are independent measurable transforms. 
Thus, by the change of variables formula, the Fubini theorem,   \eqref{eq:1-Dim}, 
and Remark \ref{consistency}
it follows that
\[ 
\begin{aligned}
&E_{x_3}\Big[ E_{x_2}\Big[E_{x_1}\Big[\mathbf{F}
\Big(\mathcal Z_{h_1}(x_1,\cdot)+\mathcal Z_{h_2}(x_2,\cdot),
     \mathcal Z_{h_3}(x_1,\cdot)+\mathcal Z_{h_4}(x_3,\cdot)\Big)\Big]\Big]\Big]\\
&=\int_{C_{0}^3[0,T]}\mathbf{F}
\Big(Y_{h_1,h_2,0}(x_1,x_2,x_3),Y_{h_3,0,h_4}(x_1,x_2,x_3)\Big)dm^3(x_1,x_2,x_3)\\
&=\int_{N}\bigg[\int_{M}\mathbf{F}(w_1,w_2)  
d(m^3\circ Y_{h_1,h_2,0}^{-1})(w_1)\bigg]d(m^3\circ Y_{h_3,0,h_4}^{-1})(w_2)\\
&=\int_{N}\bigg[\int_{C_{0}^3[0,T]}\mathbf{F}
\Big(Y_{h_1,h_2,0}(x_1,x_2,x_3),w_2\Big)dm^3(x_1,x_2,x_3)\bigg]
d(m^3\circ Y_{h_3,0,h_4}^{-1})(w_2)\\
&=\int_{N}\bigg[\int_{C_{0}^2[0,T]}\mathbf{F}
\Big(\mathcal Z_{h_1}(x_1,\cdot)+\mathcal Z_{h_2}(x_2,\cdot),w_2\Big)
dm^2(x_1,x_2)\bigg]d(m^2\circ Y_{h_3,0,h_4}^{-1})(w_2)\\
&=\int_{N}\bigg[\int_{C_{0}[0,T]}\mathbf{F}
\Big(\mathcal Z_{\mathbf{s}(h_1,h_2)}(x,\cdot),w_2\Big)dm(x)\bigg]
d(m^3\circ Y_{h_3,0,h_4}^{-1})(w_2)\\
&=\int_{C_{0}[0,T]} \bigg[\int_{N}\mathbf{F}
\Big(\mathcal Z_{\mathbf{s}(h_1,h_2)}(x,\cdot),w_2\Big) 
d(m^3\circ Y_{h_3,0,h_4}^{-1})(w_2)\bigg]dm(x)\\
&=\int_{C_{0}[0,T]} \bigg[\int_{C_{0}^3[0,T]}\mathbf{F}
\Big(\mathcal Z_{\mathbf{s}(h_1,h_2)}(x,\cdot),Y_{h_3,0,h_4}(x_1,x_2,x_3)\Big)
dm^3(x_1,x_2,x_3)\bigg]dm(x)\\
&=\int_{C_{0}[0,T]} \bigg[\int_{C_{0}^2[0,T]}\mathbf{F}
\Big(\mathcal Z_{\mathbf{s}(h_1,h_2)}(x,\cdot),  
\mathcal Z_{h_3}(x_1,\cdot) +\mathcal Z_{h_4}(x_3,\cdot)\Big)
dm^2(x_1,x_3)\bigg]dm(x)\\
&=\int_{C_{0}[0,T]}\bigg[\int_{C_{0}[0,T]}\mathbf{F}
\Big(\mathcal Z_{\mathbf{s}(h_1,h_2)}(x,\cdot),\mathcal Z_{\mathbf{s}(h_3,h_4)}(y,\cdot)\Big)
dm(y)\bigg]dm(x).
\end{aligned}
\]
Thus we obtain the desired result.
\end{proof}

\par
The following corollary is a very simple consequence 
of Theorem \ref{theorem:2-Dim-II}.

\begin{corollary}\label{theorem:2-Dim12-II-coro}
Let $h_1$,  $h_2$, $h_3$ and  $h_4$ be as 
in Theorem \ref{theorem:2-Dim-II}.
Let $F$ and $G$ be $\mu$-integrable functionals.
Then
\begin{equation}\label{eq:R-pwz4-1234-II}
\begin{aligned}
& E_{x_3}\big[E_{x_2}\big[ E_{x_1}\big[
F\big(\mathcal Z_{h_1}(x_1,\cdot)+ \mathcal Z_{h_2}(x_2,\cdot)\big)
G\big(\mathcal Z_{h_3}(x_1,\cdot)+ \mathcal Z_{h_4}(x_3,\cdot)\big) \big]\big]\big]\\
&=E_x\big[F\big(\mathcal Z_{\mathbf{s}(h_1,h_2)}(x,\cdot)\big)\big]
E_x\big[G\big(\mathcal Z_{\mathbf{s}(h_3,h_4)}(x,\cdot)\big)\big]
\end{aligned}
\end{equation}
\end{corollary}

\setcounter{equation}{0}
\section{Fourier--Feynman transforms and  convolution type operations}\label{sec:FFT+COP}

In this section we establish the facts that
the Fourier--Feynman transform of the convolution type operation is a product
of the Fourier--Feynman transforms (Theorem \ref{thm:p-main} below)
and that the convolution type operation of the Fourier--Feynman transforms is 
a product of the Fourier--Feynman transforms (Theorem \ref{thm:p-main-II} below).

\subsection{Transforms of convolution type operations}\label{sec:FFT-COP}

It will be helpful to establish the following  lemma  before giving the
first  theorem in this section.

\begin{lemma} \label{le:Fu-lemma}
Let $g_1$, $g_2$, $h_1$, $h_2$ and $k$ be nonzero functions in $BV[0,T]$.
For each $j\in\{1,2\}$,
let  $\mathbf{s}(g_jh_j,h_j)$ be the functions in $BV[0,T]$ 
satisfying  equation \eqref{eq:fn-rot} with $h_1$ and $h_2$ replaced with $g_jhj$ and $h_j$. 
Also, let  $F$ and $G$  be $\mathbb C$-valued scale invariant measurable functionals on
$C_0[0,T]$ such that the  analytic  Wiener integrals
$T_{\lambda,\mathbf{s}(g_1k,h_1)}(F)(y)$, $T_{\lambda,\mathbf{s}(g_2k,h_2)}(G)(y)$
and $(F*G)_{\lambda}^{(g_1,g_2;h_1,h_2)}(y)$
exist for every $\lambda \in \mathbb C$ and s-a.e. $y\in C_{0}[0,T]$.
Furthermore assume that given $k \in BV[0,T]\setminus\{0\}$,
the analytic transform of $(F*G)_{\lambda_1}^{(g_1,g_2;h_1,h_2)}$, 
$T_{\lambda_2, k}((F*G)_{\lambda_1}^{(g_1,g_2;h_1,h_2)})(y)$ 
exists for every $(\lambda_1,\lambda_2) \in \mathbb C_+\times \mathbb C_+$ and
s-a.e. $y\in C_{0}[0,T]$. 
Suppose $g_1g_2k^2+h_1h_2=0$.
Then  for each $\lambda\in\mathbb C_+$ and 
s-a.e. $y\in C_{0}[0,T]$,
\begin{equation}\label{eq:Fubini}
\begin{aligned}
&T_{\lambda, k}((F*G)_{\lambda}^{(g_1,g_2;h_1,h_2)})(y)\\
&=T_{\lambda,\mathbf{s}(g_1k,h_1)}(F)\big(\mathcal Z_{g_1}(y,\cdot)\big)
T_{\lambda,\mathbf{s}(g_2k,h_2)}(G)\big(\mathcal Z_{g_2}(y,\cdot)\big).\\
\end{aligned}
\end{equation}
\end{lemma}


\begin{remark}\label{re:Fu-lemma}
(\textit{Comments on the assumptions in Lemma \ref{le:Fu-lemma}})

Let a function $y\in C_0[0,T]$ be given. 
For $(\lambda_1,\rho_2)\in \mathbb C_+\times (0,+\infty)$, let
\[
\begin{aligned}
J_{\lambda_1^*}(k;\rho_2) 
:&=  T_{k,\rho_2}\big((F*G)_{\lambda_1}^{(g_1,g_2;h_1,h_2)} \big)(y)\\
&= E_{x_2}\big[(F*G)_{\lambda_1}^{(g_1,g_2;h_1,h_2)}\big(y+\rho_2^{-1/2}\mathcal Z_{k}(x_2,\cdot)\big)\big]\\
&=E_{x_2}\big[E_{x_1}^{\text{\rm an}w_{\lambda_1}}
\big[F\big(\mathcal Z_{g_1}\big(y + \rho_2^{-1/2}\mathcal Z_{k}(x_2,\cdot)\big)
+ \mathcal Z_{h_1}(x_1,\cdot) \big)\\
&\qquad\qquad\qquad\,\,\,\,\times
G\big(\mathcal Z_{g_2}\big(y +\rho_2^{-1/2}\mathcal Z_{k}(x_2,\cdot)\big)
+ \mathcal Z_{h_2}(x_1,\cdot) \big)
\big]\big]\\
&=E_{x_2}\big[E_{x_1}^{\text{\rm an}w_{\lambda_1}}
\big[F\big(\mathcal Z_{g_1}(y,\cdot) + \rho_2^{-1/2}\mathcal Z_{g_1k}(x_2,\cdot) 
+\mathcal Z_{h_1}(x_1,\cdot) \big)\\
&\qquad\qquad\qquad\,\,\,\,\times
G\big(\mathcal Z_{g_2}(y,\cdot) +\rho_2^{-1/2}\mathcal Z_{g_2k}(x_2,\cdot) 
+ \mathcal Z_{h_2}(x_1,\cdot) \big)
\big]\big],
\end{aligned}
\]
for $(\rho_1,\lambda_2)\in (0,+\infty)\times \mathbb C_+$, let
\[
\begin{aligned}
J_{\rho_1}^*(k;\lambda_2)
:&= T_{k,\lambda_2}\big((F*G)_{\rho_1}^{(g_1,g_2;h_1,h_2)} \big)(y)\\
&= E_{x_2}^{\text{\rm an}w_{\lambda_2}}
\big[(F*G)_{\rho_1}^{(g_1,g_2;h_1,h_2)}\big(y+\mathcal Z_{k}(x_2,\cdot)\big)\big]\\
&=E_{x_2}^{\text{\rm an}w_{\lambda_2}}\big[ E_{x_1}
\big[F\big(\mathcal Z_{g_1}\big(y + \mathcal Z_{k}(x_2,\cdot)\big)
+\rho_1^{-1/2}\mathcal Z_{h_1}(x_1,\cdot) \big)\\
&\qquad\quad\qquad\,\,\,\,\,\times
G\big(\mathcal Z_{g_2}\big(y +\mathcal Z_{k}(x_2,\cdot)\big)
+ \rho_1^{-1/2}\mathcal Z_{h_2}(x_1,\cdot) \big)
\big]\big]\\
&=E_{x_2}^{\text{\rm an}w_{\lambda_2}}\big[E_{x_1}
\big[F\big(\mathcal Z_{g_1}(y,\cdot) + \mathcal Z_{g_1k}(x_2,\cdot) 
+\rho_1^{-1/2}\mathcal Z_{h_1}(x_1,\cdot) \big)\\
&\qquad\quad\qquad\qquad\times
G\big(\mathcal Z_{g_2}(y,\cdot) +\mathcal Z_{g_2k}(x_2,\cdot) 
+\rho_1^{-1/2}\mathcal Z_{h_2}(x_1,\cdot) \big)
\big]\big],
\end{aligned}
\]
and for $(\rho_1,\rho_2)\in (0,+\infty)\times(0,+\infty)$, let
\[
\begin{aligned}
J_{(F,G)}(k,k;\rho_1,\rho_2)
&= T_{k,\rho_2}\big((F*G)_{\rho_1}^{(g_1,g_2;h_1,h_2)} \big)(y)\\
&= E_{x_2}\big[(F*G)_{\rho_1}^{(g_1,g_2;h_1,h_2)}\big(y+\rho_2^{-1/2}\mathcal Z_{k}(x_2,\cdot)\big)\big]\\
&=E_{x_2}\big[ E_{x_1}\big[
F\big(\mathcal Z_{g_1}\big(y + \rho_2^{-1/2}\mathcal Z_{k}(x_2,\cdot)\big)
+\rho_1^{-1/2}\mathcal Z_{h_1}(x_1,\cdot) \big)\\
&\qquad\quad\quad\quad\times
G\big(\mathcal Z_{g_2}\big(y +\rho_2^{-1/2}\mathcal Z_{k}(x_2,\cdot)\big)
+ \rho_1^{-1/2}\mathcal Z_{h_2}(x_1,\cdot) \big)
\big]\big]\\
&=E_{x_2}\big[E_{x_1}
\big[F\big(\mathcal Z_{g_1}(y,\cdot) + \rho_2^{-1/2}\mathcal Z_{g_1k}(x_2,\cdot) 
+\rho_1^{-1/2}\mathcal Z_{h_1}(x_1,\cdot) \big)\\
&\qquad\quad\quad\quad\times
G\big(\mathcal Z_{g_2}(y,\cdot) +\rho_2^{-1/2}\mathcal Z_{g_2k}(x_2,\cdot) 
+\rho_1^{-1/2}\mathcal Z_{h_2}(x_1,\cdot) \big)
\big]\big],
\end{aligned}
\]
Also, let $J_{\lambda_1^*}^*(k;\lambda_2)$,  $\lambda_2\in \mathbb C_+$,
denote  the analytic continuation of $J_{\lambda_1^*}(k;\rho_2)$,
let $J_{\lambda_1^*}^*(k;\lambda_1)$, $\lambda_1\in \mathbb C_+$,
denote  the analytic continuation of   $J_{\lambda_2^*}(h_1;\rho_1)$,
and let $J_{(F,G)}^{**}(k,k;\cdot,\cdot)$ on   $\mathbb C_+\times \mathbb C_+$,
denote  the analytic continuation of $J_{(F,G)}(k,k;\rho_1,\rho_2)$.

\par
From the   assumptions in Lemma \ref{le:Fu-lemma},
one can see that the three analytic Wiener  integrals
$J_{\lambda_1^*}^*(k;\lambda_2)$,
$J_{\lambda_1^*}^*(k;\lambda_1$, and
$J_{(F,G)}^{**}(k,k;\lambda_1,\lambda_2)$
all exist, and
\begin{equation}\label{eq:three-continuation}
J_{\lambda_1^*}^*(k;\lambda_2)
=J_{\lambda_1^*}^*(k;\lambda_1)
=J_{(F,G)}^{**}(k,k;\lambda_1,\lambda_2)
\end{equation}
for all $(\lambda_1,\lambda_2)\in\mathbb C_+\times\mathbb C_+$.
\end{remark}

\begin{proof}[Proof of Lemma \ref{le:Fu-lemma}]
In view of equations \eqref{eq:anal-T} and \eqref{eq:gcp-Z}, 
we first note that the existences of the analytic Wiener integrals
\[
T_{\lambda,\mathbf{s}(g_1h_1,h_1)}(F)(y) ,\,\,  T_{\lambda,\mathbf{s}(g_2h_2,h_2)}(G)(y) ,\,\,
(F*G)_{\lambda}^{(g_1,g_2;h_1,h_2)}(y)
\]
and
\[  
T_{\lambda_2, k}((F*G)_{\lambda_1}^{(g_1,g_2;h_1,h_2)}) (y)
\] 
guarantee that the five analytic Wiener integrals
\begin{align*}
&\,\,\mbox{(i)}\,\,\, E_{x}\big[F\big(y+\lambda^{-1/2}\mathcal Z_{\mathbf{s}(g_1k,h_1)} (x,\cdot)\big)\big],\\
&\,\mbox{(ii)} \,\,\,E_{x}\big[F\big(y+\lambda^{-1/2}\mathcal Z_{\mathbf{s}(g_2k,h_2)} (x,\cdot)\big)\big],\\
&\mbox{(iii)} \,\,\,E_{x}\big[F\big(\mathcal Z_{g_1}(y,\cdot) + \lambda^{-1/2}\mathcal Z_{h_1}(x,\cdot)  \big) 
G\big(\mathcal Z_{g_2}(y,\cdot) +\lambda^{-1/2}\mathcal Z_{h_2}(x,\cdot) \big],\\
&\mbox{(iv)} \,\,\, E_{x_2}\big[E_{x_1}
\big[F\big(\mathcal Z_{g_1}(y,\cdot) + \lambda_2^{-1/2}\mathcal Z_{g_1k}(x_2,\cdot) 
+\lambda_1^{-1/2}\mathcal Z_{h_1}(x_1,\cdot) \big)\\
&\qquad\quad\quad\quad\times
G\big(\mathcal Z_{g_2}(y,\cdot) +\lambda_2^{-1/2}\mathcal Z_{g_2k}(x_2,\cdot) 
+\lambda_1^{-1/2}\mathcal Z_{h_2}(x_1,\cdot) \big)\big]\big],\\
\intertext{and}  
&\mbox{(v)} \,\,\, E_{x_2}\big[E_{x_1}^{\text{\rm an}_{\zeta_1}}
\big[F\big(\mathcal Z_{g_1}(y,\cdot) + \zeta_2^{-1/2}\mathcal Z_{g_1k}(x_2,\cdot) 
+\mathcal Z_{h_1}(x_1,\cdot) \big)\\
&\qquad\quad\quad\quad\times
G\big(\mathcal Z_{g_2}(y,\cdot) +\zeta_2^{-1/2}\mathcal Z_{g_2k}(x_2,\cdot) 
+\mathcal Z_{h_2}(x_1,\cdot) \big)\big]\big]
\end{align*}
all exist for any $\lambda>0, \lambda_1>0$, $\lambda_2>0$,
$\zeta_1\in \mathbb C_+$,   $\zeta_2>0$, and s-a.e. $y\in C_{0}[0,T]$.

\par
Next,  the existence of the analytic Wiener  integral
\begin{equation}\label{prod-add}
\begin{aligned}
&\mathfrak J (\lambda_1,\lambda_2)   
\equiv 
T_{\lambda_2, k}((F*G)_{\lambda_1}^{(g_1,g_2;h_1,h_2)}) 
\\&=
E_{x_2}^{\text{\rm an}w_{\lambda_2}}\big[E_{x_1}^{\text{\rm an}w_{\lambda_1}}
\big[F\big(\mathcal Z_{g_1}(y,\cdot) + \mathcal Z_{g_1k}(x_2,\cdot) 
+\mathcal Z_{h_1}(x_1,\cdot) \big)\\
&\qquad\quad\qquad\qquad\quad\times
G\big(\mathcal Z_{g_2}(y,\cdot) +\mathcal Z_{g_2k}(x_2,\cdot) 
+\mathcal Z_{h_2}(x_1,\cdot) \big)
\big]\big] 
\end{aligned}
\end{equation}
also ensure that the analytic Wiener integral
\[
\begin{aligned}
\mathfrak J (\lambda,\lambda)   
&=E_{x_2}^{\text{\rm an}w_{\lambda }}\big[E_{x_1}^{\text{\rm an}w_{\lambda }}
\big[F\big(\mathcal Z_{g_1}(y,\cdot) + \mathcal Z_{g_1k}(x_2,\cdot) 
+\mathcal Z_{h_1}(x_1,\cdot) \big)\\
&\qquad\quad\qquad\qquad\quad\times
G\big(\mathcal Z_{g_2}(y,\cdot) +\mathcal Z_{g_2k}(x_2,\cdot) 
+\mathcal Z_{h_2}(x_1,\cdot) \big)
\big]\big] 
\end{aligned}
\]
is well-defined for all $\lambda\in \mathbb C_+$.
In equation \eqref{prod-add} above, by the observation in Remark \ref{re:Fu-lemma},
$\mathfrak J (\lambda_1,\lambda_2)$
means the three analytic function space integrals
in equation \eqref{eq:three-continuation}.

\par
On the other hand, using the Fubini theorem  
and \eqref{eq:R-pwz4-1234}, 
it follows that for  all $\lambda>0$ and s-a.e. $y\in C_{0}[0,T]$,
\[
\begin{aligned}
&T_{\lambda, k}((F*G)_{\lambda}^{(g_1,g_2;h_1,h_2)}) (y)
\equiv \mathfrak J (\lambda,\lambda)   \\
&=E_{x_2}\big[E_{x_1}
\big[F\big(\mathcal Z_{g_1}(y,\cdot) + \lambda^{-1/2}\mathcal Z_{g_1k}(x_2,\cdot) 
+\lambda^{-1/2}\mathcal Z_{h_1}(x_1,\cdot) \big)\\
&\qquad\quad\quad\quad\times
G\big(\mathcal Z_{g_2}(y,\cdot) +\lambda^{-1/2}\mathcal Z_{g_2k}(x_2,\cdot) 
+\lambda^{-1/2}\mathcal Z_{h_2}(x_1,\cdot) \big)\big]\big]\\
&=E_{x_2}\Big[E_{x_1}\Big[
F\Big(\mathcal Z_{g_1}(y,\cdot) + \lambda^{-1/2}\big[\mathcal Z_{g_1k}(x_2,\cdot) 
+\mathcal Z_{h_1}(x_1,\cdot)\big] \Big)\\
&\qquad\quad\quad\quad\times
G\Big(\mathcal Z_{g_2}(y,\cdot) +\lambda^{-1/2}\big[\mathcal Z_{g_2k}(x_2,\cdot) 
+ \mathcal Z_{h_2}(x_1,\cdot)\big] \Big)\Big]\Big]\\
&=E_{x}\big[F\big(\mathcal Z_{g_1}(y,\cdot) + \lambda^{-1/2} \mathcal Z_{\mathbf{s}(g_1k,h_1)}(x,\cdot)\big)\big]\\ 
&\quad\times
E_{x}\big[G\big(\mathcal Z_{g_2}(y,\cdot) + \lambda^{-1/2} \mathcal Z_{\mathbf{s}(g_2k,h_2)}(x,\cdot)\big)\big]\\ 
&=T_{\lambda, \mathcal Z_{\mathbf{s}(g_1k ,h_1)}} (F)(\mathcal Z_{g_1}(y,\cdot))
T_{\lambda, \mathcal Z_{\mathbf{s}(g_2k ,h_2)}} (G)(\mathcal Z_{g_2}(y,\cdot)).
\end{aligned}
\]
We now use the analytic continuation
to obtain our desired conclusion.
\end{proof}

\begin{theorem}  \label{thm:p-main}
Let $g_1$, $g_2$, $h_1$, $h_2$, $k$,
$\mathbf{s}(g_1h_1,h_1)$ and $\mathbf{s}(g_2h_2,h_2)$
be as in Lemma \ref{le:Fu-lemma}. 
Let $q$ be a nonzero real number 
and let $F$ and $G$  be  $\mathbb C$-valued scale invariant measurable functionals
on $C_0[0,T]$ such that the $L_1$ analytic Feynman integrals
$T_{q,\mathbf{s}(g_1k,h_1)}(F)(y)$, $T_{q,\mathbf{s}(g_2k,h_2)}(G)(y)$
and $(F*G)_{q}^{(g_1,g_2;h_1,h_2)}(y)$
exist for  s-a.e. $y\in C_{0}[0,T]$.
Furthermore assume that given $k \in BV[0,T]\setminus\{0\}$,
the analytic FFT of $(F*G)_{q}^{(g_1,g_2;h_1,h_2)}$, 
$T_{q, k}((F*G)_{q}^{(g_1,g_2;h_1,h_2)})(y)$ 
exists for s-a.e. $y\in C_{0}[0,T]$. 
Suppose 
\[
g_1g_2k^2+h_1h_2=0.
\] 
Then  for s-a.e. $y\in C_{0}[0,T]$,
\begin{equation}\label{eq:Fubini-q}
\begin{aligned}
&T_{q, k}((F*G)_{q}^{(g_1,g_2;h_1,h_2)})(y)\\
&=T_{q,\mathbf{s}(g_1k,h_1)}(F)\big(\mathcal Z_{g_1}(y,\cdot)\big)
T_{q,\mathbf{s}(g_2k,h_2)}(G)\big(\mathcal Z_{g_2}(y,\cdot)\big).\\
\end{aligned}
\end{equation}
\end{theorem}

\begin{remark} \label{re:existece-comments}
(\textit{Comments on the assumptions in Theorem  \ref{thm:p-main}})

Before giving the proof of Theorem \ref{thm:p-main},
we will emphasize the following assertions:

\par
(i) The existence conditions for 
\[
T_{q,\mathbf{s}(g_1k,h_1)}(F) , \,\, T_{q,\mathbf{s}(g_2k,h_2)}(G) \,\,
\mbox{ and } \,\, (F*G)_{q}^{(g_1,g_2;h_1,h_2)} 
\] 
say  that  
\[
T_{\lambda,\mathbf{s}(g_1k,h_1)}(F)(y),\,\, 
T_{\lambda,\mathbf{s}(g_2k,h_2)}(G)(y) 
\,\, \mbox{ and  }
\,\, 
(F*G)_{\lambda}^{(g_1,g_2;h_1,h_2)}(y)
\]
all exist for  all $\lambda  \in \mathbb C_+$ and s-a.e. $y\in C_{0}[0,T]$.

\par
(ii) The existence conditions  for $(F*G)_{q}^{(g_1,g_2;h_1,h_2)}$
and $T_{q, k}((F*G)_{q}^{(g_1,g_2;h_1,h_2)})$ say  that 
\begin{itemize}
\item 
 $T_{\lambda, k}((F*G)_{q}^{(g_1,g_2;h_1,h_2)})(y)$
 exists for every $\lambda  \in \mathbb C_+$ and s-a.e. $y\in C_{0}[0,T]$; 
and 
\item  $T_{\lambda_2, k}((F*G)_{\lambda_1}^{(g_1,g_2;h_1,h_2)})(y)$
exists for every $(\lambda_1,\lambda_2) \in \mathbb C_+\times \mathbb C_+$
and s-a.e. $y\in C_{0}[0,T]$.
\end{itemize}  

\par
Thus the assumptions in Theorem \ref{thm:p-main} involve the assumptions in 
Lemma \ref{le:Fu-lemma}.
\end{remark}

\begin{proof}[Proof of Theorem \ref{thm:p-main}]
To obtain equation \eqref{eq:Fubini-q}, one may establish that
\[
\begin{aligned}
&T_{q, k}((F*G)_{q}^{(g_1,g_2;h_1,h_2)})(y)\\
&=\lim_{\begin{subarray}{1}\lambda_2\to -iq \\
\lambda_2 \in \mathbb C_+ \end{subarray}}
E_{x_2}^{\text{\rm an}w_{\lambda_2}}
\big[(F*G)_{q}^{(g_1,g_2;h_1,h_2)} (y+\mathcal Z_{k}(x_2,\cdot))\big]\\
&=\lim_{\begin{subarray}{1}\lambda_1,\lambda_2\to -iq \\
\lambda_1,\lambda_2 \in \mathbb C_+\end{subarray}}
E_{x_2}^{\text{\rm an}w_{\lambda_2}}
\big[E_{x_1}^{\text{\rm an}w_{\lambda_1}}
\big[F\big(\mathcal Z_{g_1}(y,\cdot) +  \mathcal Z_{g_1k}(x_2,\cdot) 
 +\mathcal Z_{h_1}(x_1,\cdot) \big)\\
&\qquad\qquad\qquad\qquad\qquad\qquad\times
G\big(\mathcal Z_{g_2}(y,\cdot) +\mathcal Z_{g_2k}(x_2,\cdot) 
+ \mathcal Z_{h_2}(x_1,\cdot) \big)\big]\big],\\
&=\lim_{\begin{subarray}{1}\lambda\to -iq \\
\lambda \in \mathbb C_+\end{subarray}}
E_{x}^{\text{\rm an}w_{\lambda}}
\big[F\big(\mathcal Z_{g_1}(y,\cdot) +  \mathcal Z_{\mathbf{s}(g_1k,h_1)}(x,\cdot) \big)\big]\\
&\quad \times
 \lim_{\begin{subarray}{1}\lambda\to -iq \\
\lambda \in \mathbb C_+\end{subarray}}
E_{x}^{\text{\rm an}w_{\lambda}}
\big[G\big(\mathcal Z_{g_2}(y,\cdot) +  \mathcal Z_{\mathbf{s}(g_2k,h_2)}(x,\cdot) \big)\big]\\
&=T_{q,\mathbf{s}(g_1k,h_1)}(F)\big(\mathcal Z_g(y,\cdot)\big)
T_{q,\mathbf{s}(g_2k,h_2)}(G)\big(\mathcal Z_g(y,\cdot)\big).
\end{aligned}
\]
But, as shown in the proof of Lemma \ref{le:Fu-lemma},  the assertions 
in Remark \ref{re:existece-comments} that the  analytic  Wiener
integrals
\[
\begin{aligned}
T_{\lambda,\mathbf{s}(g_1k,h_1)}(F)(y)    
&= E_{x}^{\text{\rm an}w_{\lambda}}
\big[F\big( y                             
+  \mathcal Z_{\mathbf{s}(g_1k,h_1)}(x,\cdot) \big)\big]\\
T_{\lambda,\mathbf{s}(g_2k,h_2)}(G)(y)    
&=E_{x}^{\text{\rm an}w_{\lambda}}
\big[G\big(y                              
+  \mathcal Z_{\mathbf{s}(g_2k,h_2)}(x,\cdot) \big)\big]\\
\end{aligned}
\]
and
\[
\begin{aligned}
&(F*G)_{\lambda}^{(g_1,g_2;h_1,h_2)}(y)\\
&=E_{x}^{\text{\rm an}w_{\lambda_1}}\big[
F\big(\mathcal Z_{g_1}(y,\cdot) +  \mathcal Z_{h_1}(x,\cdot)\big)
G\big(\mathcal Z_{g_2}(y,\cdot) +  \mathcal Z_{h_2}(x,\cdot) \big)\big]\big] 
\end{aligned}
\]
exist for  every $\lambda  \in \mathbb C_+$ and s-a.e. $y\in C_{0}[0,T]$,
and  the  analytic Wiener  integral
\[  
\begin{aligned}
&T_{\lambda_2, k}((F*G)_{\lambda_1}^{(g_1,g_2;h_1,h_2)}) (y)\\
&=E_{x_2}^{\text{\rm an}w_{\lambda_2}}\big[E_{x_1}^{\text{\rm an}_{\lambda_1}}
\big[F\big(\mathcal Z_{g_1}(y,\cdot) +  \mathcal Z_{g_1k}(x_2,\cdot) 
+\mathcal Z_{h_1}(x_1,\cdot) \big)\\
&\qquad\qquad\qquad\quad\times
G\big(\mathcal Z_{g_2}(y,\cdot) + \mathcal Z_{g_2k}(x_2,\cdot) 
+\mathcal Z_{h_2}(x_1,\cdot) \big)\big]\big]
\end{aligned}
\] 
exists for every 
$(\lambda_1,\lambda_2)\in\mathbb C_+\times\mathbb C_+$,
say  the fact that 
$T_{\lambda,\mathbf{s}(g_1k,h_1)}(F)(y)$ and $T_{\lambda,\mathbf{s}(g_2k,h_2)}(G)(y)$
are analytic on $\mathbb C_+$, as   functions of $\lambda$,
and $T_{\lambda_2, k}((F*G)_{\lambda_1}^{(g_1,g_2;h_1,h_2)}) (y)$
is  analytic on $\mathbb C_+\times\mathbb C_+$, as a  function of  $(\lambda_1,\lambda_2)$.
Thus, to establish equation \eqref{eq:Fubini-q},
it will suffice  to show that
\[
\begin{aligned}
&T_{q, k}((F*G)_{q}^{(g_1,g_2;h_1,h_2)})(y)\\
&=\lim_{\begin{subarray}{1}\lambda\to -iq \\
\lambda \in \mathbb C_+\end{subarray}}
E_{x_2}^{\text{\rm an}w_{\lambda}}
\big[E_{x_1}^{\text{\rm an}w_{\lambda}}
\big[F\big(\mathcal Z_{g_1}(y,\cdot) +  \mathcal Z_{g_1k}(x_2,\cdot) 
 +\mathcal Z_{h_1}(x_1,\cdot) \big)\\
&\qquad\qquad\qquad\qquad\qquad
\times
G\big(\mathcal Z_{g_2}(y,\cdot) +\mathcal Z_{g_2k}(x_2,\cdot) 
+ \mathcal Z_{h_2}(x_1,\cdot) \big)\big]\big],\\
&=\lim_{\begin{subarray}{1}\lambda\to -iq \\
\lambda \in \mathbb C_+\end{subarray}}
E_{x}^{\text{\rm an}w_{\lambda}}
\big[F\big(\mathcal Z_{g_1}(y,\cdot) +  \mathcal Z_{\mathbf{s}(g_1k,h_1)}(x,\cdot) \big)\big]\\
&\quad \times
 \lim_{\begin{subarray}{1}\lambda\to -iq \\
\lambda \in \mathbb C_+\end{subarray}}
E_{x}^{\text{\rm an}w_{\lambda}}
\big[G\big(\mathcal Z_{g_2}(y,\cdot) +  \mathcal Z_{\mathbf{s}(g_2k,h_2)}(x,\cdot) \big)\big]\\
&=T_{q,\mathbf{s}(g_1k,h_1)}(F)\big(\mathcal Z_{g_1}(y,\cdot)\big)
T_{q,\mathbf{s}(g_2k,h_2)}(G)\big(\mathcal Z_{g_2}(y,\cdot)\big).
\end{aligned}
\]
But it follows from equation \eqref{eq:Fubini} and the analytic continuation.
\end{proof}

\subsection{Convolution type operations of transforms}\label{sec:COP-FFT}

In our second theorem  we establish the fact that
the convolution type operation of the Fourier--Feynman transforms is 
a product of the Fourier--Feynman transforms.  

\begin{theorem} \label{thm:p-main-II}
Let $g_1$, $g_2$, $k_1$, $k_2$, $h_3$ and $h_4$ be nonzero functions in $BV[0,T]$.
Let $\mathbf{s}(h_3,k_1)$ and $\mathbf{s}(h_4,k_2)$ be given as in equation 
\eqref{eq:fn-rot}.
Also, let  $F$ and $G$  be $\mathbb C$-valued scale invariant measurable functionals on
$C_0[0,T]$ such that given a real $q\in \mathbb R\setminus\{0\}$, the  analytic FFTs on $C_0[0,T]$
such that $T_{q, k_1}(F)$, $T_{q, k_2}(G)$, 
$T_{q,\mathbf{s}(g_1k,h_1)}(F)$ and
$T_{q,\mathbf{s}(g_2k,h_2)}(G)$
exist for  s-a.e. $y\in C_{0}[0,T]$.
Furthermore assume that the convolution type operation 
$(T_{q, k_1}(F)*T_{q, k_2}(G)_{\lambda_3})^{(g_1,g_2;h_3,h_4)}$
exists for s-a.e. $y\in C_{0}[0,T]$. 
Suppose 
\begin{equation}\label{support-condition}
m_L (\text{\rm supp}(h_3)\cap  \text{\rm supp}(h_4))=0.
\end{equation} 
Then  for given $q\in \mathbb R\setminus\{0\}$ and  s-a.e. $y\in C_{0}[0,T]$,
\begin{equation}\label{eq:Fubini-CT}
\begin{aligned}
&\big(T_{q, k_1}(F)*T_{q, k_2}(G)\big)_{q}^{(g_1,g_2;h_3,h_4)} (y)\\
&=T_{q,\mathbf{s}(h_3,k_1)}(F)\big(\mathcal Z_{g_1} (y,\cdot)\big)
  T_{q,\mathbf{s}(h_4,k_2)}(G)\big(\mathcal Z_{g_2} (y,\cdot)\big).\\
\end{aligned}
\end{equation}
\end{theorem}

\begin{proof}
By  similar arguments in Remark \ref{re:existece-comments} and \ref{re:Fu-lemma},
the following analytic continuation of the seven Wiener integrals
\[
\begin{aligned}
J_1(\rho_1,\rho_2,\rho_3) 
&=\big(T_{\rho_1, k_1}(F)*T_{\rho_2, k_2}(G)\big)_{\rho_3}^{(g_1,g_2;h_3,h_4)},\quad
   \rho_1,\rho_2,\rho_3 \in (0,+\infty)\\
J_2(\rho_2,\rho_3;\lambda_1) 
&=\big(T_{\lambda_1, k_1}(F)*T_{\rho_2, k_2}(G)\big)_{\rho_3}^{(g_1,g_2;h_3,h_4)},\quad
  \rho_2,\rho_3 \in  (0,+\infty),\,\lambda_1\in\mathbb C_+\\
J_3(\rho_1,\rho_3;\lambda_2) 
&=\big(T_{\rho_1, k_1}(F)*T_{\lambda_2, k_2}(G)\big)_{\rho_3}^{(g_1,g_2;h_3,h_4)},\quad
  \rho_1,\rho_3 \in  (0,+\infty),\,\lambda_2\in\mathbb C_+\\
J_4(\rho_1,\rho_2;\lambda_3) 
&=\big(T_{\rho_1, k_1}(F)*T_{\rho_2, k_2}(G)\big)_{\lambda_3}^{(g_1,g_2;h_3,h_4)},\quad
  \rho_1,\rho_2 \in  (0,+\infty),\,\lambda_3\in\mathbb C_+\\
J_5(\rho_3;\lambda_1,\lambda_2) 
&=\big(T_{\lambda_1, k_1}(F)*T_{\lambda_2, k_2}(G)\big)_{\rho_3}^{(g_1,g_2;h_3,h_4)},\quad
  \rho_3 \in  (0,+\infty),\,\lambda_1,\lambda_2\in\mathbb C_+\\
J_6(\rho_2;\lambda_1,\lambda_3) 
&=\big(T_{\lambda_1, k_1}(F)*T_{\rho_2, k_2}(G)\big)_{\lambda_3}^{(g_1,g_2;h_3,h_4)},\quad
 \rho_2 \in  (0,+\infty),\,\lambda_1,\lambda_3\in\mathbb C_+\\
\end{aligned}
\]
and
\[
\begin{aligned}
J_7(\rho_1;\lambda_2,\lambda_3) 
=\big(T_{\rho_1, k_1}(F)*T_{\lambda_2, k_2}(G)\big)_{\lambda_3}^{(g_1,g_2;h_3,h_4)},\quad
 \rho_1 \in  (0,+\infty),\,\lambda_2,\lambda_3\in\mathbb C_+\\
\end{aligned}
\]
all exist  and have the same analytic continuation (analytic Wiener integral on $C_0^3[0,T]$)
\[
J^*(\lambda_1,\lambda_2,\lambda_3) =\big(T_{\lambda_1, k_1}(F)*T_{\lambda_2, k_2}(G)\big)_{\lambda_3}^{(g_1,g_2;h_3,h_4)},\quad
 \lambda_1,\lambda_2,\lambda_3 \in \mathbb C_+.
\]
Thus, by similar arguments as in  the proofs of Lemma \ref{le:Fu-lemma}
and Theorem \ref{thm:p-main},  it will suffice to show  that 
equation \eqref{eq:Fubini-CT} holds for all $\lambda>0$ and s-a.e. $y\in C_0[0,T]$.

Using the Fubini theorem and   applying equation \eqref{eq:R-pwz4-1234-II}
with the condition \eqref{support-condition},
it follows that for  all $\lambda>0$ and s-a.e. $y\in C_{0}[0,T]$,
\[
\begin{aligned}
&\big(T_{\lambda, k_1}(F)*T_{\lambda, k_2}(G)\big)_{\lambda}^{(g_1,g_2;h_3,h_4)} (y)\\
&=E_{x_1}\big[T_{\lambda, k_1}(F)
\big(\mathcal Z_{g_1} (y,\cdot)+ \lambda^{-1/2}\mathcal Z_{h_3} (x_1,\cdot)\big) \\
&\qquad\quad\times
T_{\lambda, k_2}(G)
\big(\mathcal Z_{g_2} (y,\cdot)+ \lambda^{-1/2}\mathcal Z_{h_4} (x_1,\cdot)\big) \big]\\
&=E_{x_1}\big[E_{x_2}\big[ F 
\big(\mathcal Z_{g_1} (y,\cdot)+ \lambda^{-1/2}\mathcal Z_{h_3} (x_1,\cdot)+\lambda^{-1/2}\mathcal Z_{k_1} (x_2,\cdot)\big)\big] \\
&\qquad\quad\times
E_{x_3}\big[ G 
\big(\mathcal Z_{g_2} (y,\cdot)+ \lambda^{-1/2}\mathcal Z_{h_4} (x_1,\cdot)+\lambda^{-1/2}\mathcal Z_{k_2} (x_3,\cdot)\big)\big] \big]\\
&=E_{x_1}\Big[E_{x_2}\Big[E_{x_3}\Big[ F 
\Big(\mathcal Z_{g_1} (y,\cdot)+ \lambda^{-1/2}\big(\mathcal Z_{h_3} (x_1,\cdot)+ \mathcal Z_{k_1} (x_2,\cdot)\big)\Big)  \\
&\qquad\qquad\qquad\quad\times
  G \Big(  \mathcal Z_{g_2} (y,\cdot)+ \lambda^{-1/2}\big(\mathcal Z_{h_4} (x_1,\cdot)+ \mathcal Z_{k_2} (x_3,\cdot)\big)\Big]\Big]\Big]\\
&=E_{x} \big[ F \big(\mathcal Z_{g_1} (y,\cdot)+ \lambda^{-1/2} \mathcal Z_{\mathbf{s}(h_3,k_1)} (x ,\cdot)\big)\big]   \\
&\quad \times
  E_{x} \big[G \big(  \mathcal Z_{g_2} (y,\cdot)+\lambda^{-1/2} \mathcal Z_{\mathbf{s}(h_4,k_2)} (x ,\cdot)\big)\big]\\
&=T_{\lambda,\mathbf{s}(h_3,k_1)}(F)\big(\mathcal Z_{g_1} (y,\cdot)\big)
  T_{\lambda,\mathbf{s}(h_4,k_2)}(G)\big(\mathcal Z_{g_2} (y,\cdot)\big)\\
\end{aligned}
\]
as desired.
\end{proof}

\setcounter{equation}{0}
\section{Further result and Examples}\label{sec:example}

\par
The result in   Theorems  \ref{thm:p-main} and \ref{thm:p-main-II} above 
can be applied to  many large classes of functionals on $C_0[0,T]$.
These classes of functionals are appeared  in 
\cite{CS76-1,CS80,HPS95,HPS96,HPS97-1,JS79-II,PS98,PSS98}.

In Theorem \ref{thm:p-main}, we established that an analytic FFT with respect to a Gaussian process 
of a convolution type operation of functionals on $C_0[0,T]$ is a product of analytic FFTs 
with respect to a Gaussian processes,  and in Theorem \ref{thm:p-main-II}, 
we established that a  convolution type  operation of analytic FFTs with respect to a Gaussian processes 
is a product of analytic FFTs with respect to a Gaussian processes.

Here  we have the following question: how to relate the two results in Theorems 
\ref{thm:p-main} and \ref{thm:p-main-II}, i.e.,
how to find the conditions on the transforms and convolution type operations
in the next equation?
\begin{equation}\label{eq:p-main-I+II}
T_{q, k}((F*G)_{q}^{(g_1,g_2;h_1,h_2)})(y)
= \big(T_{q, k_1}(F)*T_{q, k_2}(G)\big)_{q}^{(g_1,g_2;h_3,h_4)} (y).
\end{equation}

In views of the assumptions in Theorems \ref{thm:p-main} and \ref{thm:p-main-II},
we must check that there exist solutions $\{g_1,g_2,k,k_1,k_2,h_1,h_2,h_3,h_4\}$ of the system
\begin{equation}\label{eq:system}
\begin{cases}
\,\, \mbox{(i)}   \,\,\, g_1g_2k^2+h_1h_2=0 \,\mbox{ in } L_2[0,T],\\
  \, \mbox{(ii)}    \,\, m_L (\text{\rm supp}(h_3)\cap  \text{\rm supp}(h_4))=0,\\
     \mbox{(iii)}   \,\, \mathbf{s}(g_1k,h_1)=\mathbf{s}(h_3,k_1) \,\mbox{ in } L_2[0,T], \\
         \qquad\qquad  \mbox{ i.e., }\, g_1^2(t)k^2(t)+h_1^2(t)=h_3^2(t)+k_1^2(t) \,\,m_L\mbox{-a.e.}\,\,\, t\in[0,T],\\
     \mbox{(iv)}   \,\,\mathbf{s}(g_2k,h_2)=\mathbf{s}(h_4,k_2) \,\mbox{ in } L_2[0,T], \\
         \qquad\qquad  \mbox{ i.e., }\, g_2^2(t)k^2(t)+h_2^2(t)=h_4^2(t)+k_2^2(t) \,\,m_L\mbox{-a.e.}\,\,\, t\in[0,T].
\end{cases}
\end{equation} 
to establish equation  \eqref{eq:p-main-I+II} above. 

Throughout the  remainder of this paper, we give  examples 
of the solution sets  of the system \eqref{eq:system}.

\begin{example}\label{ex1}
The set $\{g_1,g_2,k,k_1,k_2,h_1,h_2,h_3,h_4\}$ of functions  in $L_2[0,T]$ with 
\[
\begin{aligned}
g_1(t)&=2\cos\Big(\frac{2\pi t}{T}\Big)\chi_{[0,T/2]}(t),  
      & 
g_2(t)&=\Big[3-4\sin^2\Big(\frac{2\pi t}{T}\Big)\Big]\chi_{[T/2,T]}(t), \\
k(t)&=\sin\Big(\frac{2\pi t}{T}\Big), 
\quad 
 k_1(t)=\sin\Big(\frac{4\pi t}{T}\Big),&
 k_2(t)&=\sin\Big(\frac{6\pi t}{T}\Big),  \\
 h_1(t)&=\chi_{[T/2,T]}(t), 
&
 h_2(t)&=\chi_{[0,T/2]}(t),    \\
  h_3(t)&=\cos\Big(\frac{4\pi t}{T}\Big)\chi_{[T/2,T]}(t),   
&
 h_4(t)&=\cos\Big(\frac{6\pi t}{T}\Big)\chi_{[0,T/2]}(t), &\\
\end{aligned}
\]
is a solution of the system   \eqref{eq:system}.
The functions above are defined $m_L$-a.e. on $[0,T]$.
 \end{example}

\begin{example}\label{ex2}
Given positive integers  $l$, $m$ and $n$ with $l<m<n$,
let
\[
\begin{aligned}
g_1(t)&=\sin\Big(\frac{l\pi t}{T}\Big),  
\quad\,\,
g_2(t) =\sin\Big(\frac{m\pi t}{T}\Big), &
 &k(t)  =\cos\Big(\frac{n\pi t}{T}\Big),  &\\
 k_1(t)&=\sqrt2\sin\Big(\frac{l\pi t}{T}\Big)\cos\Big(\frac{n\pi t}{T}\Big)\chi_{B}(t),  
      & 
k_2(t)&=\sqrt2\sin\Big(\frac{m\pi t}{T}\Big)\cos\Big(\frac{n\pi t}{T}\Big)\chi_{B}(t),   \\
h_1(t)&=\sin\Big(\frac{l\pi t}{T}\Big)\cos\Big(\frac{n\pi t}{T}\Big),
      & 
h_2(t)& =-\sin\Big(\frac{m\pi t}{T}\Big)\cos\Big(\frac{n\pi t}{T}\Big),  \\
h_3(t)&=\sqrt2\sin\Big(\frac{l\pi t}{T}\Big)\cos\Big(\frac{n\pi t}{T}\Big)\chi_{A}(t),
      & 
h_4(t)& =\sqrt2\sin\Big(\frac{m\pi t}{T}\Big)\cos\Big(\frac{n\pi t}{T}\Big)\chi_{A}(t).  \\
\end{aligned}
\]
Then the set 
$\mathbf{S}= \{g_1,g_2,k,k_1,k_2,h_1,h_2,h_3,h_4\}$ is a solution of the system   \eqref{eq:system}.

In fact, the solution set  $\mathbf{S}$ can be obtained by the following procedures.
First, let $\{A,B\}$ be  a measurable partition of  $[0,T]$ with 
$m_L(A)>0$ and $m_L(B)>0$.
Next, given any functions $g_1$, $g_2$ and $k$ in $BV[0,T]$, 
let
\[
\begin{aligned}
h_1(t)&=g_1(t)k(t),
      & 
h_2(t)&=-g_2(t)k(t), \\
h_3(t)&=\sqrt2g_1(t)k(t)\chi_{A}(t),
      & 
h_4(t)&=\sqrt2g_2(t)k(t)\chi_{A}(t),  \\
k_1(t)&=\sqrt2g_1(t)k(t)\chi_{B}(t),
      & 
k_2(t)&=\sqrt2g_2(t)k(t)\chi_{B}(t). \\
\end{aligned}
\]
Then one can see that the set $\{g_1,g_2,k,h_1,h_2,h_3,h_4,k_1,k_2\}$ is a solution of the  system
\eqref{eq:system}.
 \end{example}

\begin{example}\label{ex3}
Let $\mathcal H=\{h_n \}_{n=1}^\infty$ be the sequence of Haar functions on $[0,T]$.
It is well-known that $\mathcal H$ is a complete orthonormal set on $L_2[0,T]$
which is consists of nonsmooth functions.

Consider the intervals  $A=[0,T/2]$ and $B=[T/2,T]$.
Then, for each $n\in \mathbb N$ with $n\ne 1$, either $\text{\rm supp}(h_n)\subset [0,T/2]$
or $\text{\rm supp}(h_n)\subset [T/2,T]$.

Let 
\[
P_1=\{n\in \mathbb N : \text{\rm supp}(h_n) \subset [0,T/2]\}
\]
and
\[
P_2=\{n\in \mathbb N : \text{\rm supp}(h_n)  \subset [T/2,T]\}.
\]
Then, clearly, 
\[
\bigcup_{n\in P_1} \text{\rm supp}(h_n) =[0,T/2]
\,\,\mbox{ and }\,\,
\bigcup_{n\in P_2} \text{\rm supp}(h_n) =[T/2,T].
\]
Also, let  
\[
\mathcal H^A = \{\chi_A\}\cup \{h_n/\sqrt2 : n\in P_1\}\equiv \{h_n^A\}_{n=1}^{\infty}
\]
and 
\[
\mathcal H^B = \{\chi_B\}\cup \{h_n/\sqrt2 : n\in P_2\}\equiv \{h_n^B\}_{n=1}^{\infty}.
\]
Then it follows that
\begin{itemize}
  \item[(i)] $\mathcal H^A$ is a complete orthonormal set  in $L_2(A)=L_2[0,T/2]$; and 
  \item[(ii)] $\mathcal H^B$ is a complete orthonormal set in  $L_2(B)=L_2[T/2,T]$.
\end{itemize}

As discussed in Example \ref{ex2} above, 
given $g_1$, $g_2$ and $k$ in $BV[0,T]$,  let
\begin{align*}
h_1(t)&=g_1(t)k(t)    & \mbox{ a.e. }\,t\in[0,T],\\
h_2(t)&=-g_2(t)k(t)   & \mbox{ a.e. }\,t\in[0,T].
\end{align*}
In these settings, for $j\in \{1,2\}$, let
\[
\sum_{n=1}^{\infty} \alpha_{n}^{(j)} h_n^A\equiv \sum_{n} \alpha_{n}^{(j)} h_n^A
\] 
be  the Fourier series  of $\sqrt2g_jk$ with respect to $\mathcal H^A$ on $[0,T/2]$,
and let
\[
\sum_{n=1}^{\infty} \beta_{n}^{(j)} h_n^B\equiv \sum_{n} \beta_{n}^{(j)} h_n^B
\]
be  the Fourier series of  $\sqrt2g_jk$ with respect to $\mathcal H^B$ on $[T/2,T]$.
Then, one can see that
\begin{itemize}
  \item[(i)] $g_1g_2(t)k^2(t)+h_1(t)h_2(t)=g_1g_2(t)k^2(t)-g_1g_2(t)k^2(t)=0$,
  \item[(ii)] $m_L(\text{\rm supp}(h_3)\cap  \text{\rm supp}(h_4)) =m_L(A\cap B)=0$, 
  \item[(iii)]
\[
\begin{aligned}
&g_1^2(t)k^2(t)+h_1^2(t)
 =2g_1^2(t)k^2(t)
=[\sqrt2g_1(t)k(t)]^2 \\
&=[\sqrt2g_1(t)k(t)\chi_A(t)   +\sqrt2g_1(t)k(t)\chi_B(t)]^2 \\
&= 2g_1^2(t)k^2(t)\chi_A(t)
  +2g_1^2(t)k^2(t)\chi_B(t) \\
&=\bigg(\sum_{n=1}^{\infty} \alpha_{n}^{(1)} h_n^A\bigg)^2(t)+\bigg(\sum_{n=1}^{\infty} \beta_{n}^{(1)} h_n^B \bigg)^2(t),
\end{aligned}
\]
 \item[(iv)]
\[
\begin{aligned}
&g_2^2(t)k^2(t)+h_2^2(t)
 =2g_2^2(t)k^2(t)
 =[\sqrt2g_2(t)k(t)]^2 \\
&=[\sqrt2g_2(t)k(t)\chi_A(t)   +\sqrt2g_1(t)k(t)\chi_B(t)]^2 \\
&= 2g_2^2(t)k^2(t)\chi_A(t)
  +2g_2^2(t)k^2(t)\chi_B(t) \\
&=\bigg(\sum_{n=1}^{\infty} \alpha_{n}^{(2)} h_n^A\bigg)^2(t)+\bigg(\sum_{n=1}^{\infty} \beta_{n}^{(2)} h_n^B \bigg)^2(t)
\end{aligned}
\]
\end{itemize}
for a.e. $t\in [0,T]$. Thus, given functions $g_1$, $g_2$ and $k$ in $BV[0,T]$, 
it follows that
\[
\begin{aligned}
&T_{q, k}((F*G)_{\lambda}^{(g_1,g_2;g_1k,-g_2k)})(y)\\
&=
\big(T_{q,\sum_n \beta_{n}^{(1)} h_n^B}(F)*
T_{q, \sum_n \alpha_{n}^{(2)} h_n^B}(G)\big)_{q}^{(g_1,g_2;
\sum_n \alpha_{n}^{(1)} h_n^A,\sum_n \beta_{n}^{(2)} h_n^A)}(y)
\end{aligned}
\]
for s-a.e. $y\in C_0[0,T]$.
\end{example}


\end{document}